\documentclass[oupthm,crop,info]{CUP-JNL-CJM}%

\usepackage{graphicx}
\usepackage{multicol,multirow}
\usepackage{amsmath,amssymb,amsfonts}
\usepackage{mathrsfs}
\usepackage{rotating}
\usepackage{appendix}
\usepackage[numbers]{natbib}
\usepackage{amssymb}
\usepackage{url}
\usepackage{adjustbox}
\usepackage{wasysym}
\usepackage{calc, mathdots}
\usepackage{tikz-cd}
\usepackage{mathrsfs}

	\normalbaselines						  %
\def\RR{\mathbb{R}}

\def\CC{\mathbb{C}}
\def\NN{\mathbb{N}}

\newcommand{\al}{{\alpha}}
\newcommand{\la}{{\lambda}}
\newcommand{\f}{{\varphi}}

\newcommand{\R}{{\mathbb  R}}

\newcommand{\te}{{\vartheta}}

\newcommand{\N}{{\mathbb  N}}
\newcommand{\C}{{\mathbb  C}}

\newcommand{\bA}{{\boldsymbol{A}}}
\newcommand{\bal}{{\boldsymbol{\alpha}}}

\newcommand{\dd}{{d}}

\newcommand{\fdot}{\,\cdot\,}

\makeatletter
\def\Ddots{\mathinner{\mkern1mu\raise\p@
\vbox{\kern7\p@\hbox{.}}\mkern2mu
\raise4\p@\hbox{.}\mkern2mu\raise7\p@\hbox{.}\mkern1mu}}
\makeatother

\newcommand{\cH}{\mathcal{H}}
\newcommand{\G}{\Theta}

\newcommand{\cB}{\mathcal{B}}

\newcommand{\cA}{\mathcal{A}}
\newcommand{\eps}{\varepsilon}

\DeclareMathOperator{\tr}{tr}
\newcommand{\be}{\mathbf{e}}

\newcommand{\cD}{\mathcal{D}}
\newcommand{\cG}{\mathcal{G}}

\newcommand{\fD}{\mathfrak{D}}
\newcommand{\fh}{\mathfrak{H}}
\newcommand{\fK}{\mathfrak{K}}

\newcommand{\cI}{\mathcal{I}}

\newcommand{\ft}{\mathfrak{t}}

\newcommand{\ciG}{\ci\Theta}
\newcommand{\bB}{\mathbf{B}}

\DeclareMathOperator{\Ker}{Ker}
\DeclareMathOperator{\Mul}{mul}
\DeclareMathOperator{\dom}{dom}

\DeclareMathOperator{\Ran}{ran}
\DeclareMathOperator{\spa}{span}

\DeclareMathOperator{\clos}{clos}

\newcommand{\ci}[1]{_{ {}_{\scriptstyle #1}}}
\newcommand{\ti}[1]{_{\scriptstyle \text{\rm #1}}}
\catcode`@=11
\chardef\mathlig@atcode\count255

\def\actively#1#2{\begingroup\uccode`\~=`#2\relax\uppercase{\endgroup#1~}}
\def\mathlig@gobble{\afterassignment\mathlig@next@cmd\let\mathlig@next= }
\def\mathlig@delim{\mathlig@delim}
\def\mathlig@defcs#1{\expandafter\def\csname#1\endcsname}
\def\mathlig@let@cs#1#2{\expandafter\let\expandafter#1\csname#2\endcsname}
\def\mathlig@appendcs#1#2{\expandafter\edef\csname#1\endcsname{\csname#1\endcsname#2}}

\def\mathlig#1#2{\mathlig@checklig#1\mathlig@end\mathlig@defcs{mathlig@back@#1}{#2}\ignorespaces}

\def\mathlig@checklig#1#2\mathlig@end{%
 \expandafter\ifx\csname mathlig@forw@#1\endcsname\relax
 \expandafter\mathchardef\csname mathlig@back@#1\endcsname=\mathcode`#1%
 \mathcode`#1"8000\actively\def#1{\csname mathlig@look@#1\endcsname}%
 \mathlig@dolig#1\mathlig@delim
\fi
\mathlig@checksuffix#1#2\mathlig@end
}

\def\mathlig@checksuffix#1#2\mathlig@end{%
\ifx\mathlig@delim#2\mathlig@delim\relax\else\mathlig@checksuffix@{#1}#2\mathlig@end\fi
}
\def\mathlig@checksuffix@#1#2#3\mathlig@end{%
\expandafter\ifx\csname mathlig@forw@#1#2\endcsname\relax\mathlig@dosuffix{#1}{#2}\fi
\mathlig@checksuffix{#1#2}#3\mathlig@end
}

\def\mathlig@dosuffix#1#2{%
\mathlig@appendcs{mathlig@toks@#1}{#2}%
\mathlig@dolig{#1}{#2}\mathlig@delim
}

\def\mathlig@dolig#1#2\mathlig@delim{%
 \mathlig@defcs{mathlig@look@#1#2}{%
 \mathlig@let@cs\mathlig@next{mathlig@forw@#1#2}\futurelet\mathlig@next@tok\mathlig@next}%
 \mathlig@defcs{mathlig@forw@#1#2}{%
  \mathlig@let@cs\mathlig@next{mathlig@back@#1#2}%
  \mathlig@let@cs\checker{mathlig@chck@#1#2}%
  \mathlig@let@cs\mathligtoks{mathlig@toks@#1#2}%
  \expandafter\ifx\expandafter\mathlig@delim\mathligtoks\mathlig@delim\relax\else
  \expandafter\checker\mathligtoks\mathlig@delim\fi
  \mathlig@next
 }%
 \mathlig@defcs{mathlig@toks@#1#2}{}%
 \mathlig@defcs{mathlig@chck@#1#2}##1##2\mathlig@delim{%
  \ifx\mathlig@next@tok##1%
   \mathlig@let@cs\mathlig@next@cmd{mathlig@look@#1#2##1}\let\mathlig@next\mathlig@gobble
  \fi
  \ifx\mathlig@delim##2\mathlig@delim\relax\else
   \csname mathlig@chck@#1#2\endcsname##2\mathlig@delim
  \fi
 }%
 \ifx\mathlig@delim#2\mathlig@delim\else
  \mathlig@defcs{mathlig@back@#1#2}{\csname mathlig@back@#1\endcsname #2}%
 \fi
}%

\catcode`@\mathlig@atcode

\mathchardef\ordinarycolon\mathcode`\:
\def\vcentcolon{\mathrel{\mathop\ordinarycolon}}
\mathlig{:=}{\vcentcolon=}
\mathlig{::=}{\vcentcolon\vcentcolon=}

\theoremstyle{oupplain}
\newtheorem{theo}{Theorem}[section]
\newtheorem*{theorem*}{Theorem}
\newtheorem{lem}[theo]{Lemma}
\newtheorem{cor}[theo]{Corollary}
\newtheorem{prop}[theo]{Proposition}

\theoremstyle{oupdefinition}
\newtheorem{defn}{Definition}[section]
\theoremstyle{oupremark}
\newtheorem{Remark}[theo]{Remark}
\newtheorem{ex}[theo]{Example}
\theoremstyle{oupproof}
\newtheorem{proof}{Proof}

\numberwithin{equation}{section}

\articletype{RESEARCH ARTICLE}
\jname{Canadian Mathematical Society}
\jyear{2022}
\begin{document}

\begin{Frontmatter}

\title[Singular Boundary Conditions via Perturbation Theory]{Singular Boundary Conditions for Sturm--Liouville Operators via Perturbation Theory\thanks{M.~Bush and C.~Liaw were partially supported by the National Science Foundation grant  DMS-1802682. \\
D.~Frymark was partially supported by the Swedish Foundation for Strategic Research under grant AM13-0011. \\
Since August 2020, C.~Liaw has been serving as a Program Director in the Division of Mathematical Sciences at the National Science Foundation (NSF), USA, and as a component of this position, she received support from NSF for research, which included work on this paper. Any opinions, findings, and conclusions or recommendations expressed in this material are those of the authors and do not necessarily reflect the views of the NSF.}}

\author[2]{Michael Bush}

\author[1]{Dale~Frymark}

\author[2,3]{Constanze~Liaw}

\authormark{M.~Bush, D.~Frymark and C.~Liaw}

\address[2]{\orgdiv{Department of Mathematical Sciences}, \orgname{University of Delaware}, \orgaddress{501 Ewing Hall, \city{Newark} DE 19716, \country{USA}.}
\email{mikebush@udel.edu}, \email{liaw@udel.edu}}

\address[1]{\orgdiv{Department of Theoretical Physics} \orgname{Nuclear Physics Institute, Czech Academy of Sciences}, \orgaddress{\city{Řež} 25068, \country{Czech Republic}.}
\email{frymark@ujf.cas.cz}}

\address[3]{\orgdiv{CASPER}, \orgname{Baylor University}, \orgaddress{One Bear Place \#97328, \city{Waco} TX 76798, \country{USA}.}
\email{liaw@udel.edu}}



\abstract{
We show that all self-adjoint extensions of semi-bounded Sturm--Liouville operators with limit-circle endpoint(s) can be obtained via an additive singular form bounded self-adjoint perturbation of rank equal to the deficiency indices, say $d\in\{1,2\}$. This characterization generalizes the well-known analog for semi-bounded Sturm--Liouville operators with regular endpoints. Explicitly, every self-adjoint extension of the minimal operator can be written as 
\begin{align*}
    \boldsymbol{A}\ci\Theta=\boldsymbol{A}_0+{\bf B}\Theta{\bf B}^*,
\end{align*}
where $\boldsymbol{A}_0$ is a distinguished self-adjoint extension and $\Theta$ a self-adjoint linear relation in $\CC^d$. The perturbation is singular in the sense that it does not belong to the underlying Hilbert space but is form bounded with respect to $\boldsymbol{A}_0$, i.e.~it belongs to $\cH_{-1}(\boldsymbol{A}_0)$, with possible `infinite coupling.' A boundary triple and compatible boundary pair for the symmetric operator are constructed to ensure the perturbation is well-defined and self-adjoint extensions are in a one-to-one correspondence with self-adjoint relations $\Theta$.

The merging of boundary triples with perturbation theory provides a more holistic view of the operator's matrix-valued spectral measures: identifying not just the location of the spectrum, but also certain directional information.

As an example, self-adjoint extensions of the classical Jacobi differential equation (which has two limit-circle endpoints) are obtained and their spectra are analyzed with tools both from the theory of boundary triples and perturbation theory. 
}

\end{Frontmatter}

%%%%%%%%%%%%%%%%%%%%%%%%%%%%%
%%%%%%%%%%%%%%%%%%%%%%%%%%%%%
\section{Introduction}
%%%%%%%%%%%%%%%%%%%%%%%%%%%%%
%%%%%%%%%%%%%%%%%%%%%%%%%%%%%

Perturbation theory and Sturm--Liouville operators have a rich history of interdependence; perturbation theoretic ideas are often tested on these well-known ordinary differential operators, and advances in special classes of operators (such as one-dimensional Schr\"odinger operators with various potentials) often require new perturbation results. In particular, the work of N.~Aronszajn and W.~Donoghue, responsible for well-known results on the spectra of rank-one self-adjoint perturbation problems, were motivated by applications to Sturm--Liouville differential equations \cite{Aron, Dono}. 

However, the current formulation of most perturbation results makes them only suitable for a select class of Sturm--Liouville operators---those for which any limit-circle endpoints are regular---despite the fact that many classical operators do not fall into this category. The current manuscript aims to expand the classical setup in which all self-adjoint extensions of a regular Sturm--Liouville operator are represented by an additive perturbation, to the case where the Sturm--Liouville operator can have more singular limit-circle endpoints. This leads to many new possibilities for the spectral analysis of such operators, as will be explored later. 

In order to give a frame of reference, we briefly present a slight variation of the well-known regular endpoint setup employed in \cite[Section~11.6]{S} that yields a rank-one perturbation problem.

\begin{ex} 
Let $\ell[\cdot]$ be a Sturm--Liouville operator defined by equation \eqref{d-sturmop} on $L^2[(a,b),w(x)]$ with the endpoint $a$ regular and the endpoint $b$ in the limit-point case. For simplicity, we will assume that the symmetric operator $\{\ell,\cD\ti{min}\}$ is positive for the moment. Note that any second-order differential equation with real coefficients on $[a,b]$ that are sufficiently differentiable can be expressed as a Sturm--Liouville operator \cite[Section~1.5]{N}. Such an operator requires a single boundary condition at $x=a$ which can be written, for $f\in L^2[(a,b),w(x)]$ and $0\leq\te<\pi$, as
\begin{align}\label{e-regbc}
    f(a)\cos(\te)+(pf')(a)\sin(\te)=0.
\end{align}

Hence, $\te=0$ corresponds to the Dirichlet boundary condition and $\te=\pi/2$ to the Neumann boundary condition. Let $\cD\ti{max}$ denote the maximal domain associated with $\ell$, those functions in $L^2[(a,b),w(x)]$ that $\ell$ maps back into $L^2[(a,b),w(x)]$ (see Definition~\ref{d-max}). Then, the operator $\ell_{\te}$ defined by acting via $\ell$ on the domain
\begin{align*}
    \cD_{\te}=\left\{f\in\cD\ti{max}~:~f\text{ obeys }\eqref{e-regbc}\right\},
\end{align*}
is easily seen to be self-adjoint for any fixed $\te$.

Let $f\in\cD_{\te}$ be real-valued. Then integration by parts yields

\begin{equation}\label{e-regibp}
\begin{aligned}
    \langle f,\ell_{\te}[f]\rangle&=\int_a^b f(x)\left\{\left[p(x)f'(x)\right]'+q(x)\right\} dx \\
    &=\int_a^b p(x)\left[f'(x)\right]^2 dx+\int_a^b q(x)\left[f(x)\right]^2+p(a)f(a)f'(a) \\
    &\equiv\tau(f,f)-\cot(\te)\left[f(a)\right]^2,
\end{aligned}
\end{equation}
where equation \eqref{e-regbc} and the fact that $x=b$ is in the limit-point case were used to simplify expressions and $\tau$ stands for the quadratic form defined by the first two summands in the middle line. Let $\boldsymbol{A}$ be the $\te=\pi/2$ Neumann boundary condition operator and $\delta_a$ be the Dirac delta function with singularity at $x=a$. Then
\begin{align}\label{d-aalpha}
    \boldsymbol{A}_{\al}=\boldsymbol{A}+\al\langle \cdot,\delta_a\rangle\delta_a,
\end{align}
corresponds to $\ell_{\te}=\boldsymbol{A}_{-\cot(\te)}$ via equation \eqref{e-regibp}. Since $x=a$ is a regular endpoint, the distribution $\delta_a$ gives rise to a well-defined linear functional acting on the maximal domain $\cD\ti{max}$ thanks to Sobolev embedding \cite[Section 11.6]{S}. However, $\delta_a$ is a distribution and clearly does not belong to $\cH=L^2[(a,b),w(x)]$. We conclude that the perturbation in \eqref{d-aalpha} is singular form bounded in the sense of \cite{AK}, i.e.,~$\delta_a\in\cH_{-1}(\boldsymbol{A})\setminus \cH$. See Section~\ref{ss-pert} for a definition of the scale of Hilbert spaces $\cH_{s}(\boldsymbol{A})$ for $s\in \R$ and Remark \ref{r-pertforms} for an explanation of how the case of `infinite coupling' is handled. We direct readers interested in  very general $p(x)$ or $w(x)$ to \cite{KURASOV1}.$\hfill\bullet$
\end{ex}

One of the contributions of the current manuscript is to extend equation \eqref{d-aalpha} to semi-bounded Sturm--Liouville differential equations with general limit-circle endpoint(s), see Section \ref{s-setup}. In general, such endpoints are much more difficult to handle than regular ones due to the fact that functions in the maximal domain may blow-up near the endpoints. This results in more complicated boundary conditions to ensure self-adjointness, as typical Dirichlet and Neumann boundary conditions are not rigorously defined. Recent analysis of boundary conditions and the $m$-function in the limit-circle case, in general and for examples, can be found in \cite{EGNT, EK, GLN, GZ, GZ2, KL} and further historical results can be found in \cite{BH, BG, Ful, MZ, NZ, Titch}. A direct approach, adapted from the above setup for regular endpoints, is therefore not successful. Instead, a boundary triple \cite{Bruk, Calkin, DM1, DM2, GG, Koch} and a compatible boundary pair \cite{Arl, LS} is constructed for such an operator with limit-circle endpoint(s). This framework is an application of methods developed in the recent book \cite{BdS} and we invite the reader to read the rich history of contributions to these fields in its well-written introduction. A suitable amount of background on these objects for our task can be found for convenience in Subsection \ref{ss-bt}. Some knowledge of semi-bounded forms associated with semi-bounded self-adjoint operators is assumed, although a few facts can also be found in Subsection \ref{ss-bt} and their explicit link with the perturbation setup is described in Remark \ref{r-pertforms}. Further comments on general connections between the boundary triple and boundary pair framework and setting up perturbation problems are made in Subsection \ref{ss-generalpert}.

The boundary triple and boundary pair framework introduces one difference from the classical perturbation theory that should be commented upon. The adapted rank-two perturbation setup in Theorem \ref{t-twopert} does not rely on $2\times 2$ Hermitian matrices as a parameter, which we denote $\Theta$, but rather $2\times 2$ self-adjoint relations. This divergence from the classical theory arises naturally and we usually assume that we are in the special case where the linear relation is a matrix. The price to pay for this assumption is that we are unable to claim that some spectral results apply to all self-adjoint extensions. However, there is a great deal of information able to be discerned from these matrix-valued spectral measures that scalar-valued spectral measures do not capture. In particular, separated (or local) boundary conditions, which correspond to $\Theta$ being a diagonal matrix, have simple eigenvalues when there are two limit-circle endpoints \cite{Z}, but the richer context of connected (or non-local) boundary conditions can give rise to eigenvalues of multiplicity two\footnote{The Sturm-Liouville differential equation on a loop is included as an example where connected, or non-local, boundary conditions are required.}. This discrepancy is a main motivation for the investigation and our attempts to go deeper than classical results like \cite[Theorem 10.6.3]{Z}.

Boundary triples naturally define a self-adjoint extension with the domain being functions in the kernel of the map $\Gamma_0$. This self-adjoint extension is not the unperturbed operator $\bA_0$ in our setup though. Indeed, the domain of $\bA_0$ is transversal with the natural self-adjoint extension defined by the boundary triple. Fortunately, this difference is not an impediment to the analysis of the operators, as they are intimately related. More details on how the two methods start from ``opposite'' self-adjoint extensions can be found in Remark \ref{r-0vsinfty}. Comparing and contrasting the strengths of these two methods is one of the central objectives of the paper.

Another advantage of using boundary triples is that `infinite coupling' in the perturbation parameter is easily handled, and corresponds to the cases where $\Theta$ is a $2\times 2$ self-adjoint relation but not a Hermitian matrix. Such cases technically fall outside the definition of form-bounded perturbations but can be given rigorous meaning by appealing to the semi-bounded forms associated with the relevant self-adjoint extensions, see Remark \ref{r-pertforms} for more information.

The newly expanded perturbation setup is also utilized to gain new insight about the classical Jacobi differential equation in Section \ref{s-examples}, and builds upon some of the recent results in \cite{F}. Notably, the Weyl-$m$ function for important self-adjoint extensions was determined in \cite{F} and utilized as a starting point for the spectral analysis here. Our attention is focused on this operator for many reasons:  it is extremely well-studied \cite{BEZ, B, E, FL, I, K, Sz}, contains many interesting examples as special choices of its parameters and is limit-circle at both endpoints. It is also a change of variables away from the hypergeometric differential equation (a fact which is exploited heavily), and every second-order ordinary differential equation with at most three regular singular points can be transformed into the hypergeometric differential equation. Analysis of the Jacobi equation thus practically covers all semi-bounded Sturm--Liouville equations with two limit-circle endpoints. It should be noted that such Sturm--Liouville equations, including Jacobi, are guaranteed to have only discrete spectrum \cite{BEZ, Z}.

This leaves many opportunities for further investigation into the absolutely continuous spectrum of self-adjoint extensions, which may occur when a Sturm--Liouville operator has, e.g.~one limit-circle endpoint and one limit-point endpoint. 

In cases where the perturbation parameter $\Theta$ is a Hermitian $2\times 2$ matrix or a special linear relation, several spectral properties are analyzed. The matrix-valued weights of point masses for the unperturbed operator, multiplicity of eigenvalues and eigenvectors from the space $L^2(\mu^\Theta)$, where $\mu^\Theta$ is the matrix-valued spectral measure of the perturbed operator, are all determined explicitly. (Note that the definition of the space $L^2(\mu^\Theta)$ becomes clear through its inner product given in  equation \eqref{d-matrixL^2} below.) Most of these properties are obtained by looking at explicit Weyl-$m$ functions that are features of the boundary triple framework and spectral results from \cite{GT}. Theorem \ref{t:simple} refines some observations on multiplicity/degeneracy of eigenvalues from \cite{Z}. Perturbation results from \cite{LT_JST} and \cite{LT09} further assist with the analysis of the spectrum and allow for the tracing of eigenfunctions through the perturbation in Corollary \ref{t-eigentrace}.

The spectral analysis of a wide array of operators, including one-dimensional Schr\"odinger operators with $L^1\ti{loc}$ potentials, can benefit from our results. Largely unexplored applications include higher order operators.

%%%%%%%%%%%%%%%%%%%%%%%%%%%%%
%%%%%%%%%%%%%%%%%%%%%%%%%%%%%
\subsection{Roadmap of Results}\label{ss-roadmap}
%%%%%%%%%%%%%%%%%%%%%%%%%%%%%
%%%%%%%%%%%%%%%%%%%%%%%%%%%%%

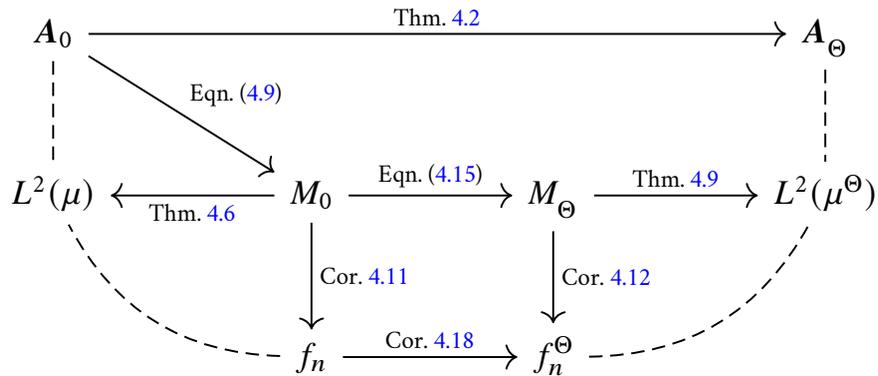
\begin{figure}[b]
\adjustbox{scale=1.3, center}{\begin{tikzcd}[row sep=large,column sep=huge]
  \bA_0 \arrow{dr}{\textrm{Eqn.~\eqref{e-weyl0}}} \arrow[d,dashed, no head] \arrow[rrr,"\textrm{Thm.~\ref{t-jacobipert}}"]
   & & & \bA\ci\G \arrow[d,dashed,no head] \\
   L^2(\mu)
   & M_0\arrow[l,"\textrm{Thm.~\ref{t:mulambdan}}"]\arrow{r}{\textrm{Eqn.~\eqref{e-MTheta}}}\arrow[d,"\textrm{Cor.~\ref{c:l2mu}}"] 
   & M\ci\G\arrow[r,"\textrm{Thm.~\ref{t:simple}}"]\arrow[d,"\textrm{Cor.~\ref{c:EVE2}}"]
   & L^2(\mu^\G) \\
   & f_n\arrow[ul,dashed,no head,bend left=30]\arrow[r,"\textrm{Cor.~\ref{t-eigentrace}}"]
   & f_n^{\G}\arrow[ur,dashed,no head,bend right=30]
   & 
   \end{tikzcd}}
   %\end{tikzpicture}
   \caption{This diagram is not to be interpreted as a commuting diagram, but rather as follows: The dashed lines represent standard connections between an operator, the space of its spectral representation, and its eigenfunctions in their spectral representation. Equation \eqref{e-weyl0}, Theorems \ref{t:mulambdan}, \ref{t:simple}, as well as Corollaries \ref{c:l2mu}, \ref{c:EVE2} and \ref{t-eigentrace} start at required information and end at information provided by the indicated equation/theorem/corollary. Only Theorem \ref{t-jacobipert} and equation \eqref{e-MTheta} contain actual maps.}
   \label{figure1}
\end{figure}

Figure~\ref{figure1} informally summarizes several key results in this paper. While some of these results are obtained via boundary triples, other are proved using perturbation theory. It is this merging of theories that allows us to obtain such a complete picture of the situation. Clearly, some of the symbols are not yet defined. And so, we refer the reader to the respective theorems, equations and corollaries to find the exact meaning of the symbols. The solid arrows in this diagram do not necessarily represent explicit maps from the objects that they originate at to the object that they terminate at. Rather, they represent the flow of influence and information.

%%%%%%%%%%%%%%%%%%%%%%%%%%%%%
%%%%%%%%%%%%%%%%%%%%%%%%%%%%%
\subsection{Notation}\label{ss-notation}
%%%%%%%%%%%%%%%%%%%%%%%%%%%%%
%%%%%%%%%%%%%%%%%%%%%%%%%%%%%

Notation such as ${\bf L}\ti{max}=\{\ell,\cD\ti{max}\}$ is used to say that the operator ${\bf L}\ti{max}$ acts via $\ell$ on the domain $\cD\ti{max}$. Operators are often written in bold to help distinguish them from linear relations where there may be confusion.

The symbol $\bal$ is used to denote the rank-one perturbation parameter, while $\al$ and $\beta$ are reserved as parameters for the Jacobi differential operator in Section \ref{s-examples}.

Finally, the word ``singular'' is used in two ways in this paper: to signify when a perturbation vector is taken from outside of the Hilbert space natural to the problem (see Subsection \ref{ss-pert}) and when an endpoint of a Sturm--Liouville equation is in the limit-circle or limit-point case (see the end of Subsection \ref{ss-extensions}). With regards to endpoints, the limit-point case is not of interest, as it does not require a boundary condition.

Note that $\infty$ is used as a subscript in Sections \ref{s-setup} and \ref{s-examples} to refer to specific self-adjoint extensions that are related to infinite coupling from the perturbation theoretic viewpoint. The most specific information can be found in Remarks \ref{r-formordering} and \ref{r-pertforms}.

%%%%%%%%%%%%%%%%%%%%%%%%%%%%%
%%%%%%%%%%%%%%%%%%%%%%%%%%%%%
\section{Background}\label{s-bg}
%%%%%%%%%%%%%%%%%%%%%%%%%%%%%
%%%%%%%%%%%%%%%%%%%%%%%%%%%%%

Consider the classical Sturm--Liouville differential equation
\begin{align}\label{d-sturmdif}
\dfrac{d}{dx}\left[p(x)\dfrac{df}{dx}(x)\right]+q(x)f(x)=-\lambda w(x)f(x),
\end{align}
where $p(x),w(x)>0$ a.e.~on $(a,b)$ and $q(x)$ real-valued a.e.~on $(a,b)$, with $a,b\in\RR\cup\{\pm \infty\}$.
Furthermore, let $1/p(x),q(x),w(x)\in L^1\ti{loc}[(a,b),dx]$. Additional details about Sturm--Liouville theory can be found, e.g.~in \cite{AG, BEZ, E, GZ, Z}.
The differential expression can be viewed as a linear operator, mapping a function $f$ to the function $\ell[f]$ via
\begin{align}\label{d-sturmop}
\ell[f](x):=-\dfrac{1}{w(x)}\left(\dfrac{d}{dx}\left[p(x)\dfrac{df}{dx}(x)\right]+q(x)f(x)\right).
\end{align}
This unbounded operator acts on the Hilbert space $L^2[(a,b),w]$, endowed with the inner product 
$
\langle f,g\rangle:=\int_a^b f(x)\overline{g(x)}w(x)dx.
$
In this setting, the eigenvalue problem $\ell[f](x)=\lambda f(x)$ can be considered. However, the operator acting via $\ell[\fdot]$ on $L^2[(a,b),w]$ is in general not self-adjoint a priori; additional boundary conditions may be  required to ensure this property. 

We assume throughout the manuscript that both $a$ and $b$ are non-oscillatory, see e.g.~\cite{Z} for more.

%%%%%%%%%%%%%%%%%%%%%%%%%%%%%
%%%%%%%%%%%%%%%%%%%%%%%%%%%%%
\subsection{Extension Theory}\label{ss-extensions}
%%%%%%%%%%%%%%%%%%%%%%%%%%%%%
%%%%%%%%%%%%%%%%%%%%%%%%%%%%%

We begin by presenting the classical self-adjoint extension theory for symmetric operators as applied to ordinary differential operators. The classical references \cite{AG,N} should be consulted for further details. Let $\ell$ be a Sturm--Liouville differential expression. It is important to reiterate that the analysis of self-adjoint extensions does not involve changing the differential expression associated with the operator at all, merely the domain of definition by applying boundary conditions.

\begin{defn}[{\cite[Section 17.2]{N}}]\label{d-max}
The {\bf maximal domain} of $\ell[\fdot]$ is given by 
\begin{align*}
\cD\ti{max}=\cD\ti{max}(\ell):=\left\{f:(a,b)\to\mathbb{C}~:~f,pf'\in\text{AC}\ti{loc}(a,b);
f,\ell[f]\in L^2[(a,b),w]\right\}.
\end{align*}
\end{defn}

The designation of ``maximal'' is appropriate in this case because $\cD\ti{max}(\ell)$ is the largest possible subset of functions from the Hilbert space on which the application of $\ell$ makes sense and that is mapped back into the Hilbert space by this application. For $f,g\in\cD\ti{max}(\ell)$ and $a<a_0\le b_0<b$ the {\bf sesquilinear form} associated with $\ell$ is given by 
\begin{equation}\label{e-greens}
[f,g](b_0)-[f,g](a_0):=\int_{a_0}^{b_0}\left\{\ell[f(x)]\overline{g(x)}-\ell[\overline{g(x)}]f(x)\right\}w(x)dx.
\end{equation}
Each of these values on the left-hand side can be seen to exist individually (and not just as a difference) by taking limits. 
\begin{theo}[{\cite[Section 17.2]{N}}]\label{t-limits}
The limits $[f,g](b):=\lim_{x\to b^-}[f,g](x)$ and $[f,g](a):=\lim_{x\to a^+}[f,g](x)$ exist and are finite for $f,g\in\cD\ti{max}(\ell)$.
\end{theo}

The equation \eqref{e-greens} is {\bf Green's formula} for $\ell[\fdot]$, and in the case of Sturm--Liouville operators it can be explicitly computed using integration by parts to be the modified Wronskian
\begin{align}\label{e-mwronskian}
[f,g]\bigg|_a^b:=p(x)[f'(x)g(x)-f(x)g'(x)]\bigg|_a^b.
\end{align}

\begin{defn}[{\cite[Section 17.2]{N}}]\label{d-min}
The {\bf minimal domain} of $\ell[\fdot]$ is given by
\begin{align*}
\cD\ti{min}=\cD\ti{min}(\ell):=\left\{f\in\cD\ti{max}(\ell)~:~[f,g]\big|_a^b=0~~\forall g\in\cD\ti{max}(\ell)\right\}.
\end{align*}
\end{defn}

The maximal and minimal operators associated with the expression $\ell[\fdot]$ are then defined as ${\bf L}\ti{min}=\{\ell,\cD\ti{min}\}$ and ${\bf L}\ti{max}=\{\ell,\cD\ti{max}\}$ respectively. By {\cite[Section 17.2]{N}}, these operators are adjoints of one another, i.e.~$({\bf L}\ti{min})^*={\bf L}\ti{max}$ and $({\bf L}\ti{max})^*={\bf L}\ti{min}$. The operator ${\bf L}\ti{min}$ is thus symmetric.

Note that the self-adjoint extensions of a symmetric operator coincide with those of the closure of the symmetric operator {\cite[Theorem XII.4.8]{DS}}, so without loss of generality we assume that ${\bf L}\ti{min}$ is closed.

\begin{defn}[variation of {\cite[Section 14.2]{N}}]\label{d-defect}
Define the {\bf positive defect space} and the {\bf negative defect space}, respectively, by
$$\cD_+:=\left\{f\in\cD\ti{max}~:~{\bf L}\ti{max}f=if\right\}
\qquad\text{and}\qquad
\cD_-:=\left\{f\in\cD\ti{max}~:~{\bf L}\ti{max}f=-if\right\}.$$
\end{defn}

The dimensions dim$(\cD_+)=m_+$ and dim$(\cD_-)=m_-$, called the {\bf positive} and {\bf negative deficiency indices of ${\bf L}\ti{min}$} respectively, will play an important role. They are usually conveyed as the pair $(m_+,m_-)$. The symmetric operator ${\bf L}\ti{min}$ has self-adjoint extensions if and only if its deficiency indices are equal {\cite[Section 14.8.8]{N}}.

The classical {\bf von Neumann formula} then says that 
\begin{equation}\label{e-vN}
\cD\ti{max}=\cD\ti{min}\dotplus\cD_+\dotplus\cD_-.\nonumber
\end{equation}
The decomposition can be made into an orthogonal direct sum by using the graph norm, see \cite{FFL}.

If the operator ${\bf L}\ti{min}$ has any self-adjoint extensions, then the deficiency indices of ${\bf L}\ti{min}$ have the form $(m,m)$, where $0\leq m\leq 2$ {\cite[Section 14.8.8]{N}}.
Hence, Sturm--Liouville expressions that generate self-adjoint operators must have deficiency indices $(0,0)$, $(1,1)$ or $(2,2)$. If a differential expression is either in the limit-circle case or regular at the endpoint $a$, it requires a boundary condition at $a$. If it is in the limit-point case at the endpoint $a$, it does not require a boundary condition. The analogous statements are true at the endpoint $b$.

Sturm--Liouville differential expressions are rather well-understood, see e.g.~\cite{BEZ,E} for encyclopedic references, so the deficiency indices are well-known in almost all cases of interest. Representative examples are the Jacobi operator, with deficiency indices $(2,2)$, the Laguerre operator, with indices $(1,1)$, and the Hermite operator, with indices $(0,0)$. The sources \cite{BEZ, GZ, NZ, W, Z} can also be consulted for more information on the classification of endpoints.

%%%%%%%%%%%%%%%%%%%%%%%%%%%%%
%%%%%%%%%%%%%%%%%%%%%%%%%%%%%
\subsection{Boundary Triples and Pairs}\label{ss-bt}
%%%%%%%%%%%%%%%%%%%%%%%%%%%%%
%%%%%%%%%%%%%%%%%%%%%%%%%%%%%

Boundary triples and pairs provide additional ways to tackle boundary value problems, and are intimately connected. Most of the material from this subsection is taken from the excellent open-source book \cite{BdS}, which should be consulted for more details and the history of contributions. The presentation here is a bit simpler, as it is formulated for operators instead of the more general linear relations.

\begin{defn}{\cite[Definition 2.1.1]{BdS}}\label{d-bt}
Let $S$ be a closed symmetric operator in a Hilbert space $\cH$. Then $\{\cG,\Gamma_0,\Gamma_1\}$ is a \textbf{boundary triple} for $S^*$ if $\cG$ is a Hilbert space and $\Gamma_0,\Gamma_1:\dom S^*\to\cG$ are surjective linear maps and the Green's identity 
\begin{align}\label{e-btgreens}
    \langle S^*f,g\rangle_{\cH}-\langle f,S^*g\rangle_{\cH}=\langle\Gamma_1f,\Gamma_0g\rangle_{\cG}-\langle\Gamma_0f,\Gamma_1g\rangle_{\cG},
\end{align}
holds for all $f,g\in \dom S^*$.
\end{defn}

Notice that when $S$ is a Sturm--Liouville differential operator the left-hand side of equation \eqref{e-btgreens} is just the sesquilinear form given in equation \eqref{e-greens} when $a_0=a$ and $b_0=b$. 

\begin{defn}{\cite[Definition 1.4.9]{BdS}}
Let $S$ be a symmetric operator in $\cH$ and let $\la\in\CC$. The space
\begin{align*}
    \mathfrak{N}_{\la}(S^*):=\ker(S^*-\la),
\end{align*}
is called the {\bf defect subspace} of $S$ at the point $\la\in\CC$.
\end{defn}

The usual positive and negative defect spaces from Definition \ref{d-defect} are then simply $\mathfrak{N}_i({\bf L}\ti{max})$ and $\mathfrak{N}_{-i}({\bf L}\ti{max})$, respectively. It should be clear that boundary triples are usually not unique. Indeed, given a self-adjoint extension $H$ of a closed symmetric operator $S$, the adjoint $S^*$ can be decomposed into a direct sum of $\dom(H)$ and a defect space. This allows maps $\Gamma_0$ and $\Gamma_1$ to be defined so that $\dom(H)=\ker(\Gamma_0)$ and $\{\cG,\Gamma_0,\Gamma_1\}$ form a boundary triple, see \cite[Theorem 2.4.1]{BdS}. The structure also allows the definition of a Weyl $m$-function, an important object for spectral theory.

\begin{defn}{\cite[Definition 2.3.1, 2.3.4]{BdS}}\label{d-mfunction}
Let $S$ be a closed symmetric operator in a complex Hilbert space $\cH$, let $\{\cG,\Gamma_0,\Gamma_1\}$ be a boundary triple for $S^*$, and let $A_0=\ker \Gamma_0$. Then
\begin{align*}
    \rho(A_0)\ni\la\mapsto M(\la)=\Gamma_1\left(\Gamma_0\upharpoonright\mathfrak{N}_{\la}(S^*)\right)^{-1},
\end{align*}
where $\upharpoonright$ denotes the restriction, is called the \textbf{Weyl $m$-function} associated with the boundary triple $\{\cG,\Gamma_0,\Gamma_1\}$.
\end{defn}

Another way that boundary triples are useful is that they provide a description of all self-adjoint extensions of closed symmetric operators with finite deficiency indices in terms of linear relations.

\begin{defn}{\cite[Section 1.1]{BdS}}
Let $\fh$ and $\fK$ be Hilbert spaces over $\CC$. A linear subspace of $\fh\times\fK$ is called a {\bf linear relation} $H$ from $\fh$ to $\fK$ and the elements $\widehat{h}\in H$ will in general be written as pairs $\{h,h'\}$ with components $h\in\fh$ and $h'\in\fK$. If $\fh=\fK$ then we will just say $H$ is a linear relation in $\fh$. \end{defn}

Let $\{\CC^{m},\Gamma_0,\Gamma_1\}$ be a boundary triple for a symmetric minimal operator $\boldsymbol{A}\ti{min}$, defined on $\cD\ti{min}$, with deficiency indices $(m,m)$. Then self-adjoint extensions $\boldsymbol{A}\ci{\Theta}\subset \boldsymbol{A}\ti{max}$ are in one-to-one correspondence with the self-adjoint relations $\Theta$ in $\CC^{m}$ via
\begin{align*}
    \dom \left(\boldsymbol{A}\ci{\Theta}\right)=\left\{f\in\cD\ti{max}~:~\{\Gamma_0 f,\Gamma_1 f\}\in\Theta\right\}.
\end{align*}
The self-adjoint linear relation can be made more practical using a decomposition such as equation \eqref{e-relationdecomp}. In the special case that $\te$ is an $m\times m$ matrix, the boundary condition for the domain of $\boldsymbol{A}\ci{\Theta}$ can be written as
\begin{align}
    \dom \left(\boldsymbol{A}\ci{\Theta}\right)=\left\{ f\in\cD\ti{max} ~:~ \Theta\Gamma_0 f=\Gamma_1 f\right\}.\nonumber
\end{align}

Let $S$ be a semi-bounded self-adjoint operator with lower bound $m(S)$. Recall that there is a natural way to identify $S$ with a closed semi-bounded form $\ft$ in $\cH$ with lower bound $m(\ft)=m(S)$ via the First and Second Representation Theorems, see \cite[Theorem 5.1.18]{BdS} and \cite[Theorem 5.1.23]{BdS}. Let $\f\in\dom (S)$, $\psi\in\dom(\ft)$ and $\gamma<m(S)$, where
\begin{align*}
    \dom(\ft_S)=\dom(S-\gamma)^{1/2}, \\
    \ft_S[\f,\psi]=\langle (S-\gamma)\f,\psi\rangle+\gamma\langle\f,\psi\rangle.
\end{align*}
Equivalently, if $\f,\psi\in\dom(\ft_S)$, then
\begin{align*}
    \ft_S[\f,\psi]=\langle (S-\gamma)^{1/2}\f,(S-\gamma)^{1/2}\psi\rangle +\gamma\langle\f,\psi\rangle.
\end{align*}
The space $\dom(\ft_S)$ endowed with the inner product
\begin{align*}
    \langle\f,\psi\rangle\ci{\ft_{S-\gamma}}:=\ft_S[\f,\psi]-\gamma\langle\f,\psi\rangle, \text{ for }\f,\psi\in\dom(\ft),
\end{align*}
is a Hilbert space, denoted $\cH\ci{\ft_{S-\gamma}}$.

More general semi-bounded symmetric operators $S$ (i.e.~the minimal operator of a Sturm--Liouville expression) also have a form associated with them via 
\begin{align}\label{e-semiformgen}
    \ft_S[f,g]=\langle Sf,g\rangle, \text{ for }f,g\in\dom(S).
\end{align}
The semi-bounded self-adjoint operator $S_F$ associated with the closure of the form $\ft_S$ in equation \eqref{e-semiformgen} is called the {\bf Friedrichs extension} of $S$, see \cite[Definition 5.3.2]{BdS}.

\begin{defn}{\cite[Definition 1.7.6]{BdS}}\label{d-transversal}
Let $S$ be a closed symmetric operator in $\cH$. If $H$ and $K$ are closed extensions such that $\dom(S)\subset\dom(H),\dom(K)\subset\dom(S^*)$, then they are called {\bf disjoint} if $\dom(H)\cap \dom(K)=\dom(S)$ and {\bf transversal} if they are disjoint and $\dom(H)\cup\dom(K)=\dom(S^*)$.
\end{defn}

Finally, we are able to state the definition of boundary pair. These pairs will help access the semi-bounded closed forms associated with extensions when a boundary triple has already been constructed. 

\begin{defn}{\cite[Definition 5.6.1]{BdS}}
Let $S$ be a closed symmetric operator in $\cH$ and let $S_1$ be a semi-bounded self-adjoint extension of $S$ with lower bound $m(S_1)$ such that $S_1$ and the Friedrichs extension $S_F$ are transversal. Let $\ft\ci{S_1}$ be the closed form associated with $S_1$ and let 
\begin{align*}
    \cH\ci{\ft_{S_1-\gamma}}=\left(\dom\left(\ft\ci{S_1}\right),\langle\cdot,\cdot\rangle\ci{\ft_{S_1-\gamma}}\right), \quad\gamma<m(S_1),
\end{align*}
be the corresponding Hilbert space. A pair $\left\{\cG,\Lambda\right\}$ is called a {\bf boundary pair} for $S$ corresponding to $S_1$ if $\cG$ is a Hilbert space and $\Lambda\in{\bf B}\left(\cH\ci{\ft_{S_1-\gamma}},\cG\right)$ satisfies
\begin{align*}
    \ker(\Lambda)=\dom\left(\ft\ci{S_F}\right) ~\text{ and }~\text{ran}(\Lambda)=\cG.
\end{align*}
\end{defn}

In order to formulate when boundary triples and pairs are related, it is necessary to first define what it means to extend an operation in our context.

\begin{defn}{\cite[Definition 5.6.3]{BdS}}
Let $S$ be a closed semi-bounded operator in $\cH$ and $\left\{\cG,\Gamma_0,\Gamma_1\right\}$ be a boundary triple for $S^*$ with $A_0=\ker(\Gamma_0)$. Let $S_1$ be a semi-bounded self-adjoint extension of $S$ such that $S_1$ and $S_F$ are transversal, so that, $\dom(S^*)\subset\dom\left(\ft\ci{S_1}\right)$. Then an operator $\Lambda\in{\bf B}\left(\cH\ci{\ft_{S_1-\gamma}},\cG\right)$ is said to be an {\bf extension} of $\Gamma_0$ if 
\begin{align*}
    \Gamma_0f=\Lambda f ~\text{ for all }f\in\dom(S^*).
\end{align*}
\end{defn}

\begin{defn}{\cite[Definition 5.6.4]{BdS}}
Let $S$ be a closed semi-bounded operator in $\cH$ and let $\left\{\cG,\Gamma_0,\Gamma_1\right\}$ be a boundary triple for $S^*$ with 
\begin{align*}
    A_0=\ker(\Gamma_0) ~\text{ and }~ A_1=\ker(\Gamma_1).
\end{align*}
Let $S_1$ be a semi-bounded self-adjoint extension of $S$ such that $S_1$ and $S_F$ are transversal and let $\left\{\cG,\Lambda\right\}$ be a boundary pair for $S$ corresponding to $S_1$. Then $\left\{\cG,\Gamma_0,\Gamma_1\right\}$ and $\left\{\cG,\Lambda\right\}$ are said to be {\bf compatible} if $\Lambda$ is an extension of $\Gamma_0$ and the self-adjoint relations $A_1$ and $S_1$ coincide.
\end{defn}

Lemma 5.6.5 of \cite{BdS} adds that in the special case when the defect indices of $S$ are finite and $\Lambda\in{\bf B}\left(\cH\ci{\ft_{S_1-\gamma}},\cG\right)$ is an extension of $\Gamma_0$, then $A_0=\ker(\Gamma_0)=S_F$ and $\left\{\cG,\Lambda\right\}$ is a boundary pair for $S$ corresponding to $S_1$. If, in this case, also $A_1$ and $S_1$ coincide then $\left\{\cG,\Gamma_0,\Gamma_1\right\}$ and $\left\{\cG,\Lambda\right\}$ are compatible. 

As will be shown in the Section \ref{s-setup}, there is a natural way to create a boundary triple for Sturm--Liouville operators. A compatible boundary pair will then help set up a perturbation problem by accessing the form domain $\ft_{S_1-\gamma}$ and defining a bounded map $\Lambda$ that acts like $\Gamma_0$ when applied to functions from $\dom(S^*)$.

%%%%%%%%%%%%%%%%%%%%%%%%%%%%
%%%%%%%%%%%%%%%%%%%%%%%%%%%%
\subsection{Finite Rank Perturbation Theory}\label{ss-pert}
%%%%%%%%%%%%%%%%%%%%%%%%%%%%
%%%%%%%%%%%%%%%%%%%%%%%%%%%%

Let $\boldsymbol{A}$ and $\boldsymbol{T}$ be (possibly unbounded) self-adjoint operators on a separable Hilbert space $\cH$. In general perturbation theory, we are interested in the following question: If we know the properties of operator $\boldsymbol{A}$ well, what can we say about the formal operator $\boldsymbol{A}+\boldsymbol{T}$? In our application to Sturm--Liouville operators, both $\boldsymbol{A}$ and $\boldsymbol{T}$ will be unbounded, though $\boldsymbol{A}$ will be bounded from below and $\boldsymbol{T}$ will be of finite rank, i.e.~has finite dimensional range. 

On the side, we note that the theory of rank-one perturbations has been flourishing for several decades. Most prominent are the results by Aronszajn \cite{Aron} and Donoghue \cite{Dono} on the mututal singularity of singular parts, by Aleksandrov \cite{Aleksandrov} on averaging/disintegration of the spectral measures, and by Poltoratski \cite{NONTAN} on the existence of non-tangential boundary values almost everywhere with respect to the singular parts. Recently, progress has been made on the finer properties of the singular parts by one of the current authors and Treil, see \cite{LTSurvey} for a survey. An application of rank-one perturbations lead to the famous Simon--Wolff criterion \cite{SimonWolff} restricting the singular parts of the discrete random Schr\"odinger operator.

Returning to our agenda of rigorously defining finite rank perturbations, by the closed graph theorem, we know that the unbounded operators $\boldsymbol{A}$ and $\boldsymbol{T}$ are only defined on a dense subset of $\cH$. If the operators are completely unrelated, it can happen that the intersection of their domains is empty. In this case, $\boldsymbol{A}+\boldsymbol{T}$ would be meaningless. In our application to  Sturm--Liouville operators, it turns out that $\boldsymbol{T}$ is relatively bounded with respect to $\boldsymbol{A}$, so that the domain of $\boldsymbol{A}+\boldsymbol{T}$ equals that of $\boldsymbol{A}$.
And the fact that $\boldsymbol{T}$ is of finite rank, allows us to formulate this perturbation problem rather concretely. To do so, we introduce a scale of Hilbert spaces associated with $\boldsymbol{A}$. In the definition of these finite rank singular form bounded perturbations roughly follows that in \cite{AK}. See also \cite{S2,S}.

Consider the non-negative operator $|\boldsymbol{A}|=(\boldsymbol{A}^*\boldsymbol{A})^{1/2}$, whose domain coincides with the domain of $\boldsymbol{A}$. Alternatively, if $\boldsymbol{A}$ is bounded from below, the shifted operator $\boldsymbol{A}+kI$, $k\in\RR$ sufficiently large, will provide a non-negative operator. We introduce a scale of Hilbert spaces.

\begin{defn}[\hspace{-1pt}{\cite[Section 1.2.2]{AK}}]\label{d-standardscale}
For $s\geq 0$, define the space $\cH_s(A)$ to consist of $\f\in\cH$ for which the $s$-norm
\begin{align}\label{d-scalenorm}
\|\f\|_s:=\|(|\boldsymbol{A}|+I)^{s/2}\f\|_{\cH},
\end{align}
is bounded. 
The space $\cH_s(\boldsymbol{A})$ equipped with the norm $||\cdot||_s$ is complete. The adjoint spaces, formed by taking the linear bounded functionals on $\cH_s(\boldsymbol{A})$, are used to define these spaces for negative indices, i.e.~$\cH_{-s}(\boldsymbol{A}):=\cH_s^*(\boldsymbol{A})$. The corresponding norm in the space $\cH_ {-s}(\boldsymbol{A})$ is thus defined by \eqref{d-scalenorm} as well. 
The collection of these $\cH_s(\boldsymbol{A})$ spaces will be called the \emph{scale of Hilbert spaces associated with the self-adjoint operator $\boldsymbol{A}$}.
\end{defn}

It is not difficult to see that the spaces satisfy the nesting properties
\begin{align*}
\hdots\subset\cH_2(\boldsymbol{A})\subset\cH_1(\boldsymbol{A})\subset\cH=\cH_0(\boldsymbol{A})\subset\cH_{-1}(\boldsymbol{A})\subset\cH_{-2}(\boldsymbol{A})\subset\hdots,
\end{align*}
and that for every two $s,t$ with $s<t$, the space $\cH_t(\boldsymbol{A})$ is dense in $\cH_s(\boldsymbol{A})$ with respect to the norm $||\cdot||_s$. Indeed, the operator $(\boldsymbol{A}+1)^{t/2}$ defines an isometry from $\cH_s(\boldsymbol{A})$ to $\cH_{s-t}(\boldsymbol{A})$. For $\psi\in\cH_s(\boldsymbol{A})$, $\f\in\cH_{-s}(\boldsymbol{A})$, we define the duality pairing
\begin{align}
\langle\psi,\f\rangle_{s}:=\left \langle(|\boldsymbol{A}|+I)^{s/2}\psi,(|\boldsymbol{A}|+I)^{-s/2}\f\right \rangle_{\cH},\nonumber
\end{align}
where the brackets on the right hand side denote the inner product of $\cH$.

Throughout the literature of other fields similar constructions occur under different names. For instance, the pairing of $\cH_1(\boldsymbol{A})$, $\cH$, and $\cH_{-1}(\boldsymbol{A})$ is sometimes referred to as a \emph{Gelfand triple} or \emph{rigged Hilbert space}. Also, when $\bA$ is the derivative operator, these scales are simply Hilbert--Sobolev spaces. When $\boldsymbol{A}$ is a general  differential operator, they are closely related. More details about Hilbert scales can be found in \cite{KP}.

It is worth noting that these Hilbert scales are related to those generated by so-called left-definite theory \cite{LW02}. This theory employs powers of a semi-bounded self-adjoint differential operator to create a continuum of operators whereupon spectral properties can be studied. The theory can be applied to self-adjoint extensions of self-adjoint operators, which can be viewed as finite rank perturbations, see e.g.~\cite{FFL, FL} and the references therein.

Rank-one perturbations of a given operator $\bA$ arise most commonly when the vectors $\f$ are bounded linear functionals on the domain of the operator $\boldsymbol{A}$; so, many applications are focused on $\cH_{-2}(\boldsymbol{A})$. Here, we  discuss the case $\f\in\cH_{-1}(\boldsymbol{A})$ for the sake of simplicity, the so-called form bounded singular case. However, references usually contain information on extensions to $\f\in\cH_{-2}(\boldsymbol{A})$, and information on the case when $\f\notin\cH_{-2}(\boldsymbol{A})$ can be found in \cite{DKS, Kurasov}.

The formulation of finite rank perturbations also uses the Hilbert scale $\cH_s(\bA).$ Our applications, again, fall into the category of form bounded, i.e.~from $\cH_{-1}(\boldsymbol{A})$, finite rank perturbations\footnote{For some results on form unbounded $\cH_{-2}(\boldsymbol{A})$ finite rank perturbations the reader may refer to \cite{AK}.}. To define rank $n$ form bounded perturbations of a self-adjoint operator $\bA$ on a Hilbert space $\cH$, consider a coordinate mapping ${\bf B}:\CC^n\to\Ran({\bf B})\subset\cH_{-1}\left({\boldsymbol {A}}\right)$ that acts via multiplication by the row vector
\begin{align*}
    \begin{pmatrix}
    f_1, ~~\hdots, f_n\\
    \end{pmatrix}
    \qquad\text{with } f_1, \hdots, f_n\in \cH_{-1}\left({\boldsymbol {A}}\right).
\end{align*}
Formally, the mapping ${\bf B}^*:\Ran({\bf B})\to\CC^n$ acts by
\begin{align*}
    {\bf B}^*\fdot=
    \begin{pmatrix}
    \langle \fdot,f_1\rangle_1 \\
    \vdots\\
    \langle \fdot,f_n\rangle_1
    \end{pmatrix}.
\end{align*}
We say \emph{formally}, because the inner products occurring in ${\bf B}^*$ are not defined on all of $\Ran({\bf B})$. However, in accordance with the definition of $\cH_{-1}\left({\boldsymbol {A}}\right)$, they do make sense as a duality pairing on the quadratic form space of the unperturbed operator $\bA$, which is all we require. Abusing notation slightly, we use the same notation ${\bf B}^*$ for the operator restricted to this form domain.

The quadratic form sense now gives rigorous meaning to the finite rank form bounded singular perturbation
\begin{align}\label{d-athetapert}
    \bA\ci\Theta:=\bA+{\bf B}\Theta{\bf B}^*,\nonumber
\end{align}
where $\Theta:\C^n\to\C^n$ is an $n\times n$ matrix; in Remark~\ref{r-pertforms} we discuss the extension of this definition to the case when $\Theta$ is a linear relation.

%%%%%%%%%%%%%%%%%%%%%%%%%%%%%
%%%%%%%%%%%%%%%%%%%%%%%%%%%%%
\section{Formulating Singular Perturbations}\label{s-setup}
%%%%%%%%%%%%%%%%%%%%%%%%%%%%%
%%%%%%%%%%%%%%%%%%%%%%%%%%%%% 

We apply the theory of boundary triples and boundary pair framework to set up a singular self-adjoint perturbation that can describe all self-adjoint extensions of a semi-bounded symmetric Sturm--Liouville operator with limit-circle endpoints. Afterwards, in Subsection \ref{ss-generalpert}, we discuss how this setup could in principle be expanded to a wider range of semi-bounded symmetric operators.

%%%%%%%%%%%%%%%%%%%%%%%%%%%%%
%%%%%%%%%%%%%%%%%%%%%%%%%%%%%
\subsection{One Limit-Circle Endpoint}\label{ss-onelc}
%%%%%%%%%%%%%%%%%%%%%%%%%%%%%
%%%%%%%%%%%%%%%%%%%%%%%%%%%%% 

Let $\ell[\cdot]$ be a Sturm--Liouville operator as in Definition \ref{d-sturmop} on the Hilbert space $\cH=L^2[(a,b),w(x)]$ such that the minimal operator ${\bf L}\ti{min}$ is semi-bounded. Further assume that the endpoint $a$ is in the limit-circle case and the endpoint $b$ is in the limit-point case. Hence, one boundary condition must be imposed on the maximal domain in order to create a  domain for a self-adjoint extension of the minimal operator ${\bf L}\ti{min}$. In order to determine these extensions, it is necessary to first look at solutions.

\begin{defn}{\cite[Definition 6.10.3]{BdS}}\label{d-principal}
Let $({\bf L}\ti{max}-\lambda_0)f=0$ with $\lambda_0\in\RR$ be non-oscillatory at the endpoint $a$ and let $u$ and $v$ be real solutions of $({\bf L}\ti{max}-\lambda_0)f=0$. Then $u$ is said to be {\bf principal} at $a$ if $1/(pu^2)$ is not integrable at $a$ and $v$ is said to be {\bf non-principal} at $a$ if $1/(pv^2)$ is integrable at $a$, where $p$ comes from the definition of $\ell$ in equation \eqref{d-sturmop}. 
\end{defn}

Let $u_a$ and $v_a$ be such a principal and non-principal solution, respectively, corresponding to the expression $\ell$ and $\la_0\in\RR$ at the endpoint $a$. A principal solution exists---due to the endpoint $a$ being in the limit-circle case---and it is uniquely determined up to real multiples. Hence, the second linearly independent solution is automatically non-principal \cite[Corollary 6.10.5]{BdS}. We also assume that these solutions are normalized in the sense that $[u_a,v_a](a)=1$, where the sesquilinear form is defined via equation \eqref{e-mwronskian}. These solutions define two operations on functions from the maximal domain via their action through the sesquilinear form. For $f\in\cD\ti{max}$, let
\begin{equation}\label{e-quasiderivs}
\begin{aligned}
    f^{[0]}(a)&:=[f,u_a](a)=\lim_{x\to a^+}\dfrac{f(x)}{v_a(x)}, \\
    f^{[1]}(a)&:=[f,v_a](a)=\lim_{x\to a^+}\dfrac{f(x)-f^{[0]}(a)v_a(x)}{u_a(x)}.
\end{aligned}
\end{equation}
Note that the simplifications of the sesquilinear forms were only recently verified for limit-circle endpoints, see \cite[Theorem 4.5]{GLN}, and that the operations are well defined due to Theorem \ref{t-limits}.
A description of all possible self-adjoint extensions of ${\bf L}\ti{min}$ is readily available in the literature, see e.g.~\cite[Theorem 3.11(ii)]{GLN}, via a one-variable parameterization. The boundary condition must be of the form, for $0\leq\te<\pi$,
\begin{align}\label{e-theta}
    [f,u_a](a)\cos(\te)+[f,v_a](a)\sin(\te)=0.
\end{align}
The choice $\te=0$ corresponds to the important Friedrich's extension via \cite[Theorem 13]{MZ}, while the choice $\te=\pi/2$ is a Krein type extension under certain conditions, see \cite[Section 5.4]{BdS}. Define ${\bf L}_{\te}$ to be the operator that acts via $\ell$ on the domain
\begin{align}\label{e-thetadomain}
    \cD_{\te}:=\left\{f\in\cD\ti{max} ~:~ f \text{ obeys }\eqref{e-theta}\right\}.
\end{align}
The operations defined in equation \eqref{e-quasiderivs} naturally form a boundary triple.

\begin{prop}{\cite[Proposition 6.12.1]{BdS}}\label{p-btone}
Assume that the endpoint $a$ is in the limit-circle case and that the endpoint $b$ is in the limit-point case. Let $v_a$ be solution of the equation $({\bf L}\ti{max}-\la_0)f=0$, $\la_0\in\RR$, which is non-principal at $a$. Then $\{\CC,\Gamma_0,\Gamma_1\}$, where
\begin{align*}
    \Gamma_0f=f^{[0]}(a) ~\text{ and }~ \Gamma_1f=-f^{[1]}(a), \quad \text{for }f\in\cD\ti{max},
\end{align*}
is a boundary triple for $\cD\ti{max}$.
\end{prop}

Let $u_b$ be a solution of $({\bf L}\ti{max}-\lambda_0)f=0$ which is principal at $b$, where $\la_0$ is the same real number used in the choice of $u_a$ and $v_a$ above. We then choose $a_0$ and $b_0$ such that $v_a$ does not vanish on $(a,a_0)$ and $u_b$ does not vanish on $(b_0,b)$ and consider $[c,d]\subset(a,b)$ with
\begin{align}\label{e-intervals}
    a<c<a_0<b_0<d<b.
\end{align}
Let $\f$ be a real solution of $({\bf L}\ti{max}-\la_0)f=0$ that is nonoscillatory at $a$ and does not vanish on $(a,a_0)$, such as $v_a$ above. We introduce the first-order differential operator that acts on functions $f\in AC(a,a_0)$ via
\begin{align*}
    N_{\f} f:=\sqrt{p}\f\left(\dfrac{f}{\f}\right)'.
\end{align*}
Note that when the potential $q=0$ in the definition of $\ell$, it may be convenient to choose $\lambda_0=0$ and $\f=1$ so that $N_{\f}f=\sqrt{p}f'$. Next, define the subset (closed under addition and scalar multiplication) of $\cH$
\begin{equation}\label{e-formdomainone}
    \begin{aligned}
    \fD:=\{f\in\cH~:~f\in AC(a,b), &\sqrt{p}f'\in L^2(c,d), \\
    &N_{v_a}f\in L^2(a,c), N_{u_b}f\in L^2(d,b)\}.
    \end{aligned} \nonumber
\end{equation}
By its definition, set $\fD$ contains those functions from the Hilbert space which have some regularity on the interior of the interval, and which are mapped `into $L^2$ near the endpoints' by the first-order operators $N_{v_a}$ and $N_{v_b}$. And by \cite[Lemma 6.12.3]{BdS} we have $\cD\ti{max}\subset\fD$.

For $f,g\in\fD$, define the form 
\begin{equation}\label{e-formone}
    \begin{aligned}
    \ft[f,g]:=\int_a^c &\left(N_{v_a}\right)f(x)\overline{\left(N_{v_a}g\right)(x)}dx+\int_d^b\left(N_{u_b}\right)f(x)\overline{\left(N_{u_b}g\right)(x)}dx \\
    &+\lambda_0\int_a^c f(x)\overline{g(x)}w(x)dx+\lambda_0\int_d^bf(x)\overline{g(x)}w(x)dx \\
    &+\int_c^d\left((\sqrt{p}f')(x)\overline{(\sqrt{p}f'(x)}+q(x)f(x)\overline{g(x)}\right)dx \\
    &+\dfrac{pv_a'(c)}{v_a(c)}f(c)\overline{g(c)}-\dfrac{pu_b'(d)}{u_b(d)}f(d)\overline{g(d)}.
    \end{aligned} \nonumber
\end{equation}
Observe that 
\begin{align*}
    \dfrac{pv_a'(c)}{v_a(c)}f(c)\overline{g(c)}\,,\dfrac{pu_b'(d)}{u_b(d)}f(d)\overline{g(d)}\in\RR,
\end{align*}
and the form $\ft$ is a densely defined closed semi-bounded form in $\cH$ by \cite[Corollary 6.12.2]{BdS}. Hence, there is a semi-bounded self-adjoint operator corresponding to $\ft$ via the First Representation Theorem \cite[Theorem 5.1.18]{BdS}, say ${\bf S}_1$. Since $\cD\ti{max}\subset\fD$, it follows that ${\bf S}_1$ is transversal with ${\bf L}_0$, the Friedrichs extension of ${\bf L}\ti{min}$ \cite[Theorem 5.3.8]{BdS}. We extend one of the mappings from the boundary triple in Proposition \ref{p-btone} to correspond to this larger domain.

Define $\Lambda:\fD\to\CC$ as
\begin{align*}
    \Lambda f:=\lim_{x\to a^+}\dfrac{f(x)}{v_a(x)}, ~~f\in\fD.
\end{align*}
The conditions that $f\in AC(a,a_0)$ and $N_{v_a}f\in L^2(a,c)$ from the definition of $\fD$ are enough to ensure that the operation $\Lambda$ is well-defined on its domain \cite[Theorem 6.10.9(i)]{BdS}.

\begin{lem}{\cite[Lemma 6.12.5]{BdS}}\label{bpone}
Let $v_a$ and $u_b$ be solutions of $({\bf L}\ti{max}-\la_0)f=0$, $\la_0=0$, which are non-principal at $a$ and principal at $b$, respectively. Assume that the endpoint $a$ is in the limit-circle case and the endpoint $b$ is in the limit-point case, and let $\{\CC,\Gamma_0,\Gamma_1\}$ be the boundary triple defined in Proposition \ref{p-btone}. Then $\{\CC,\Lambda\}$ is a boundary pair for ${\bf L}\ti{min}$ corresponding to ${\bf S}_1$ which is compatible with the boundary triple $\{\CC,\Gamma_0,\Gamma_1\}$. Moreover, one has
\begin{align*}
    \langle{\bf L}\ti{max}f,g\rangle=\langle\Gamma_1 f,\Lambda g\rangle_{\CC}+\ft[f,g], ~~f\in\cD\ti{max},~g\in\fD. 
\end{align*}
\end{lem}
Notice that the lemma shows that $\ft$ is the semi-bounded closed form associated with the semi-bounded self-adjoint operator ${\bf L}_{\pi/2}$, and thereby identifies ${\bf L}_{\pi/2}={\bf S}_1$.

It remains only to transcribe the above boundary pair into the language of perturbation theory, appealing to equation \eqref{e-theta}. Define the weighted distribution $\widetilde{\delta}_a$ that mimics the operation $\Lambda$ via
\begin{align}
\langle f,\widetilde{\delta}_a\rangle_{1}=\lim_{x\to a^+}\dfrac{f(x)}{v_a(x)}, \nonumber
\end{align}
on the domain $\fD$. The starting self-adjoint operator that will be perturbed is thus ${\bf L}_{\pi/2}$ and corresponds to $\boldsymbol{A}_0$ in the following theorem.

\begin{theo}\label{t-onepert}
Let $\te\in[0,\pi/2)$ and $\bal=-\cot(\te)$ in order to define the operator $\boldsymbol{A}_{\al}={\bf L}_{\te}$ on the domain $\dom(\boldsymbol{A}_{\bal})=\dom({\bf L}_{\te})$ as in equation \eqref{e-thetadomain}. Then $\boldsymbol{A}_{\bal}$ can be written as the singular rank-one perturbation:
\begin{align}\label{e-rankonesingular}
    \boldsymbol{A}_{\bal}=\boldsymbol{A}_0+\bal\langle\cdot,\widetilde{\delta}_a\rangle_{1}\widetilde{\delta}_a, \text{ for }\al\in\RR\cup\{\infty\}.
\end{align}
The case $\bal=\infty$ is understood to be ${\bf L}_0$, because $\vartheta = 0$. In particular, $\widetilde{\delta}_a\in\cH_{-1}(\boldsymbol{A})$ and every self-adjoint extension of the minimal operator ${\bf L}\ti{min}$ can be written as $\boldsymbol{A}_{\al}$ for some $\bal$.
\end{theo}

\begin{proof} 
It is only necessary to show that $\widetilde{\delta}_a\in\cH_{-1}(\boldsymbol{A}_0)$ and that every self-adjoint extension of ${\bf L}\ti{min}$ can be realized as ${\bf \boldsymbol{A}}_{\bal}$ for some $\bal$.
The statement $\widetilde{\delta}_a\in\cH_{-1}(\boldsymbol{A}_0)$ follows immediately from \cite[Lemma 6.12.4]{BdS} and the fact that both $\Lambda$ and $\ft$, the semi-bounded closed form associated with $\boldsymbol{A}_0$, are defined on the same domain: $\fD$. Equation \eqref{e-theta} and \cite[Theorem 3.11(ii)]{GLN} then shows that every self-adjoint extension of ${\bf L}\ti{min}$ can be realized by equation \eqref{e-rankonesingular} as $\bal$ varies.
\end{proof}
 
\begin{Remark}
We briefly comment that the case $\bal=\infty$ is rigorous in the sense of semi-bounded forms as it represents the linear relation $\{0\}\times\CC$, but this `infinite coupling' is not technically form-bounded. As this is the only relevant linear relation in this case, we postpone further discussion until Remark \ref{r-pertforms}, where the situation is more complicated. The reader should also consult \cite{AK, GS} for more on infinite coupling in the context of perturbation theory. $\hfill\bullet$
\end{Remark}

%%%%%%%%%%%%%%%%%%%%%%%%%%%%%
%%%%%%%%%%%%%%%%%%%%%%%%%%%%%
\subsection{Two Limit-Circle Endpoints}\label{ss-twolc}
%%%%%%%%%%%%%%%%%%%%%%%%%%%%%
%%%%%%%%%%%%%%%%%%%%%%%%%%%%% 

The results of Subsection \ref{ss-onelc} can easily be extended to the case of Sturm-Liouville expressions $\ell$ on the Hilbert space $\cH=L^2[(a,b),w(x)]$ such that the minimal operator ${\bf L}\ti{min}$ is semi-bounded and both endpoints are in the limit-circle case. Due to the close connection with the previous subsection, we recall only a few major differences in the setup and refer the reader to \cite[Section 6.11]{BdS} for more details about creating a boundary pair compatible with the natural boundary triple.

Let $v_a$ and $v_b$ be solutions of $({\bf L}\ti{max}-\lambda_0)f=0$, $\la_0\in\RR$, which are non-principal at $a$ and $b$, respectively. The natural boundary triple $\{\CC^2,\Gamma_0,\Gamma_1\}$ for ${\bf L}\ti{max}$ defined by these solutions is given in \cite[Proposition 6.11.1]{BdS} as
\begin{align}\label{e-bttwo}
    \Gamma_0f:=\left(\begin{array}{c}
\lim_{x\to a^+}\dfrac{f(x)}{v_a(x)}  \\
\lim_{x\to b^-}\dfrac{f(x)}{v_b(x)}
\end{array} \right), \hspace{.5cm}
    \Gamma_1f:=\left( \begin{array}{c}
-\left[f,v_a\right](a)   \\
\left[f,v_b\right](b)   
\end{array} \right) \quad \text{for }
f\in\dom \cD\ti{max}.
\end{align}

Let $a_0$ and $b_0$ be such that $v_a$ does not vanish on $(a,a_0)$, $v_b$ does not vanish on $(b_0,b)$ and $c,d$ are as in equation \eqref{e-intervals}. The closed semi-bounded form associated with the operator ${\bf L}_0$ (boundary condition $\Gamma_1f=0$ above) is then given by
\begin{equation}\label{e-formdomainbase}
\begin{aligned}
    \ft[f,g]:=\int_a^c &\left(N_{v_a}\right)f(x)\overline{\left(N_{v_a}g\right)(x)}dx+\int_d^b\left(N_{v_b}\right)f(x)\overline{\left(N_{v_b}g\right)(x)}dx \\
    &+\lambda_0\int_a^c f(x)\overline{g(x)}w(x)dx+\lambda_0\int_d^bf(x)\overline{g(x)}w(x)dx \\
    &+\int_c^d\left((\sqrt{p}f')(x)\overline{(\sqrt{p}f'(x)}+q(x)f(x)\overline{g(x)}\right)dx \\
    &+\dfrac{pv_a'(c)}{v_a(c)}f(c)\overline{g(c)}-\dfrac{pu_b'(d)}{u_b(d)}f(d)\overline{g(d)},
\end{aligned}
\end{equation}

for $f,g\in\fD$, with 
\begin{equation}\label{e-formdomain}
\begin{aligned}
    \fD:=\{f\in\cH~:~f\in AC(a,b), &\sqrt{p}f'\in L^2(c,d), \\
    &N_{v_a}f\in L^2(a,c), N_{v_b}f\in L^2(d,b)\}.
    \end{aligned}
\end{equation}
The mapping $\Gamma_0$ is thus extended as
\begin{align}\label{e-lambdabase}
\Lambda f:=\left(\begin{array}{c}
\lim_{x\to a^+}\dfrac{f(x)}{v_a(x)}  \\
\lim_{x\to b^-}\dfrac{f(x)}{v_b(x)}
\end{array} \right)
\quad \text{for } f\in\fD,
\end{align}
and $\{\CC^2,\Lambda\}$ is found to be a boundary pair for ${\bf L}\ti{min}$ compatible with the boundary triple $\{\CC^2,\Gamma_0,\Gamma_1\}$, see \cite[Lemma 6.11.5]{BdS}.

Transcribing the above boundary pair into the language of perturbation theory is analogous to the previous subsection, but requires a more careful consideration of the involved mappings. Begin by defining the weighted distributions $\widetilde{\delta}_a$ and $\widetilde{\delta}_b$ that mimic the operation $\Lambda$ via
\begin{equation}\label{e-twodist}
\begin{aligned}
\langle f,\widetilde{\delta}_a\rangle_{1}&=\lim_{x\to a^+}\dfrac{f(x)}{v_a(x)}, \\
\langle f,\widetilde{\delta}_b\rangle_{1}&=\lim_{x\to b^-}\dfrac{f(x)}{v_b(x)},
\end{aligned}\nonumber
\end{equation}
on the domain $\fD$. 

Identify the operator $\boldsymbol{A}_0$ as the self-adjoint extension of ${\bf L}\ti{min}$ that acts via $\ell[\cdot]$ on the domain
\begin{align*}
\dom(\boldsymbol{A}_0)=\left\{f\in\cD\ti{max}~:~f\in\ker\left(\Gamma_1\right)\right\},
\end{align*}
with notation chosen so as to connect to the unperturbed operator in the perturbation setup. As in Subsection \ref{ss-onelc}, we abuse notation and denote the other natural self-adjoint extension as 
\begin{align*}
\dom(\boldsymbol{A}_{\infty})=\left\{f\in\cD\ti{max}~:~f\in\ker\left(\Gamma_0\right)\right\}.
\end{align*}

\begin{Remark}\label{r-formordering}
There is an ordering of the the semi-bounded self-adjoint extensions of symmetric operators that is inherited from their form domains and plays an important role here. Namely, let $H_1$ and $H_2$ be  two semi-bounded self-adjoint operators with lower bound $a$. Let $x<a$. By \cite[Prop.~5.2.7]{BdS}, if $H_1\leq H_2$ then
\begin{align*}
    \dom(H_2-x)^{1/2}\subset \dom(H_1-x)^{1/2}.
\end{align*}
Of course, these domains are simply the domains of the semi-bounded closed forms associated with $H_1$ and $H_2$, denoted $\ft_1$ and $\ft_2$, respectively. The Friedrichs extension, $S_F$, of a semi-bounded symmetric operator, $S$, is in this sense the self-adjoint extension with the smallest form domain; if $H$ is another semi-bounded self-adjoint extension of $S$ then $H\leq S_F$, or equivalently, 
\begin{align*}
    \dom\left(\ft\ci{S_F}\right)\subset\dom(\ft_H),
\end{align*}
after an appropriate shift \cite[Prop.~5.3.6]{BdS}. In our example, the Friedrichs extension of ${\bf L}\ti{min}$ is easily identified as $\boldsymbol{A}_{\infty}$ by \cite[Cor.~6.11.9]{BdS}.

Setting up the perturbation, we would like the form domain of the unperturbed operator to be as large as possible for two reasons: the perturbation will be singular form-bounded and all self-adjoint extensions can be defined from it. This self-adjoint extension with the largest form domain is transversal (see Defn.~\ref{d-transversal}) with the Friedrichs extension \cite[Cor.~5.3.9]{BdS}. If $H$ and $S_F$ are transversal, then every semi-bounded self-adjoint extension of $S$, $H'$ satisfies 
\begin{align*}
    \dom(\ft_{H'})\subset\dom(\ft_{H}),
\end{align*}
for an appropriate shift of the operators. The operator $\boldsymbol{A}_0$ is such a transversal operator by construction in our example, and hence plays the role of the unperturbed operator. The effect of this choice of unperturbed operator can be seen later in Remark \ref{r-pertforms}, where the form interpretations of the perturbation are all well-defined. 
\hfill$\bullet$
\end{Remark}

The weighted distributions $\widetilde{\delta}_a$ and $\widetilde{\delta}_b$ are immediately seen to be in $\cH_{-1}(\boldsymbol{A}_0)$ by \cite[Theorem 6.11.4]{BdS}.

Leaning on our discussion in Subsection \ref{ss-pert}, we define a coordinate mapping ${\bf B}:\CC^2\to\Ran({\bf B})\subset\cH_{-1}\left({\boldsymbol {A}}_0\right)$ to act via multiplication by the row vector
$    
\begin{pmatrix}
    \widetilde{\delta}_a, \widetilde{\delta}_b
\end{pmatrix}.
$
A simple calculation yields that on $\cH_1(\boldsymbol{A}_0)$ the operator ${\bf B}^*$ is given by
\begin{align*}
    {\bf B}^*f=
    \begin{pmatrix}
    \langle f,\widetilde{\delta}_a\rangle_{1} \\
    \langle f,\widetilde{\delta}_b\rangle_{1}
    \end{pmatrix}.
\end{align*}

For more general information on the role these operations play in the finite rank perturbation theory of self-adjoint operators, please consult e.g.~\cite{FL2,LT_JST}.

\begin{theo}\label{t-twopert}
Let $\Theta$ be a self-adjoint linear relation in $\CC^2$. Define $\boldsymbol{A}\ci\Theta$ as the singular rank-two perturbation:
\begin{align}\label{e-twopert}
    \boldsymbol{A}\ci\Theta:=\boldsymbol{A}_0+{\bf B}\Theta{\bf B}^*.
\end{align}
Then every self-adjoint extension of the minimal operator ${\bf L}\ti{min}$ can be written as $\boldsymbol{A}\ci\Theta$ for some $\Theta$.
\end{theo}

\begin{Remark}\label{r-pertforms}
We note that the perturbation in equation \eqref{e-twopert} should be interpreted in terms of semi-bounded forms for the relevant self-adjoint extensions; this provides a rigorous interpretation of the many ways in which `infinite coupling' can occur in this perturbation setup as expressed by linear relations. In order for the perturbation to be well-defined some care should be taken so that $\Ran(\bB^*)\subset\dom(\Theta)$. In particular, the closed semi-bounded form $\ft \ci\Theta$ in $\cH$ is given by 
\begin{align}\label{e-formpert}
    \langle\boldsymbol{A}\ci\Theta f,g\rangle=\ft \ci\Theta[f,g], ~~f\in\dom(\boldsymbol{A}\ci\Theta),~g\in\dom(\ft\ci\Theta)
\end{align}
and can be expressed in terms of $\ft$ and $\Lambda$ from equations \eqref{e-formdomainbase} and \eqref{e-lambdabase}, respectively, via \cite[Theorem 6.11.6]{BdS}:
\begin{itemize}
    \item[(i)] If $\Theta$ is a $2\times 2$ symmetric matrix, then 
    \begin{align*}
        \ft \ci\Theta[f,g]=\ft[f,g]+\langle\Theta\Lambda f,\Lambda g\rangle ,~~\dom(\ft\ci\Theta)=\fD.
    \end{align*}
    \item[(ii)] If $\Theta=\Theta\ti{op}\widehat{\oplus}\Theta\ti{mul}$ (component-wise direct sum) with respect to the decomposition $\CC^2=\dom(\Theta\ti{op})+\Mul(\Theta)$ and $\dim\dom(\Theta\ti{op})=1$, then
    \begin{align*}
        \ft \ci\Theta[f,g]=\ft[f,g]+\langle\Theta\ti{op}\Lambda f,\Lambda g\rangle ,~~\dom(\ft\ci\Theta)=\left\{h\in\fD~:~\Lambda h\in\dom(\Theta\ti{op})\right\}.
    \end{align*}
    \item[(iii)] If $\Theta=\{0\}\times\CC^2$, then $\boldsymbol{A}\ci\Theta=\boldsymbol{A}_{\infty}$ coincides with the Friedrichs extension of ${\bf L}\ti{min}$ and 
    \begin{align*}
        \ft \ci\Theta[f,g]=\ft[f,g],~~\dom(\ft\ci\Theta)=\left\{h\in\fD~:~\Lambda h=0\right\}.
    \end{align*}
\end{itemize}
This interpretation also adds additional clarity to the regular endpoint example in equation \eqref{e-regibp} and the special case of Theorem \ref{t-onepert} when $\bal=\infty$.
\hfill$\bullet$
\end{Remark}

\begin{proof}
The boundary triple given in equation \eqref{e-bttwo} immediately implies that all self-adjoint extensions of ${\bf L}\ti{min}$ are in a one-to-one correspondence with self-adjoint relations in $\CC^2$ via
\begin{align*}
    \dom\boldsymbol{A}\ci{\Theta}=\left\{f\in\cD\ti{max}~:~\left\{\Gamma_0f,\Gamma_1f\right\}\in\Theta\right\}.
\end{align*}
The densely defined closed semi-bounded form $\ft\ci\Theta$ corresponding to $\boldsymbol{A}\ci\Theta$ in equation \eqref{e-formpert} also exactly matches the action of the perturbation in equation \eqref{e-twopert}.
\end{proof}

%%%%%%%%%%%%%%%%%%%%%%%%%%%%%
%%%%%%%%%%%%%%%%%%%%%%%%%%%%%
\subsection{General Perturbation Setups with Boundary Triples/Pairs}\label{ss-generalpert}
%%%%%%%%%%%%%%%%%%%%%%%%%%%%%
%%%%%%%%%%%%%%%%%%%%%%%%%%%%% 

The current manuscript deals with semi-bounded Sturm--Liouville differential operators, but there are only a few impediments to achieving a perturbation characterization like that of Theorem \ref{t-twopert} for other classes of closed symmetric semi-bounded relations. In practice, these obstacles may be very difficult to overcome in different settings, but we outline the theory here anyway.

Let $S$ be a closed symmetric semi-bounded relation in a Hilbert space $\cH$ and define a boundary triple $\left\{\cG,\Gamma_0,\Gamma_1\right\}$ for $S^*$. Denote $A_0=\ker\Gamma_0$ and $A_1=\ker\Gamma_1$. Let $S_1$ be a semi-bounded self-adjoint extension of $S$ such that $S_1$ and the Friedrichs extension of $S$, $S_F$, are transversal (see Remark \ref{r-formordering} for more). 

If $\Lambda\in{\bf B}\left(\cH\ci{\ft_{S_1-\gamma}},\cG\right)$, $\gamma<m(S_1)$, is an extension of $\Gamma_0$, $\dom\ft\ci{S_F}=\ker\Lambda$ and $A_1=S_1$, then $A_0=S_F$ and the boundary triple $\left\{\cG,\Gamma_0,\Gamma_1\right\}$ is compatible with the boundary pair $\left\{\cG,\Lambda\right\}$ by \cite[Lemma 5.6.5]{BdS}. The identification of self-adjoint extensions of $S$ and self-adjoint relations $\Theta$ in $\cG$ via
\begin{align*}
    \dom \left(\boldsymbol{A}\ci{\Theta}\right)=\left\{f\in\cD\ti{max}~:~\{\Gamma_0 f,\Gamma_1 f\}\in\Theta\right\},
\end{align*}
thus immediately leads to a potential perturbation characterization like that of Theorem \ref{t-twopert}.

The first obstacle to such a characterization is explicitly determining the Friedrichs extension of the symmetric operator, but this has been accomplished for various operators. For instance, let $\Omega\subset\RR^n$ be a bounded $C^2$ domain and $\boldsymbol{A}_D$ the self-adjoint Dirichlet realization of $-\Delta+V$ in $L^2(\Omega)$, with $V\geq 0$. Then, \cite[Corollary 8.5.4]{BdS} shows that $\boldsymbol{A}_D$ is the Friedrichs extension of the associated minimal domain. Additionally, a way of obtaining the Friedrichs extension of powers of Sturm--Liouville operators is thought to lie with the so-called general left-definite theory from \cite{LW02}.

The second obstacle is showing that $\Lambda\in{\bf B}\left(\cH\ci{\ft_{S_1-\gamma}},\cG\right)$, which means that $\Lambda\in\cH_{-1}(S_1)$. Fortunately, \cite[Theorem 5.6.6]{BdS} says that if it is known that $\ker\Gamma_0=S_F$ and $\ker\Gamma_1=S_1$ then $\Lambda\in{\bf B}\left(\cH\ci{\ft_{S_1-\gamma}},\cG\right)$. Hence, if the boundary triple is carefully set up and we know the Friedrichs extension then this obstacle is easily overcome. It may occasionally be of practical importance to determine the form $\ft_{S_1-\gamma}$ explicitly, but this is a separate problem.

Finally, the last obstacle is a bit more esoteric in the sense that such a characterization may not be useful unless there are sufficient perturbation theoretic techniques to take advantage of. Subsections \ref{ss-pert}, \ref{ss-Eigen} and \ref{ss-V} provide references and spectral results that prove to be helpful when the deficiency indices of the relation $S$ are finite and a finite rank perturbation is formulated. On the other hand, it is also convenient to work with Weyl $m$-functions which come straight from the boundary triple to gain spectral information. We pursue this path in most of Section \ref{s-examples} (from the beginning through Corollary \ref{c:EVE2}).

%%%%%%%%%%%%%%%%%%%%%%%%%%%%%
%%%%%%%%%%%%%%%%%%%%%%%%%%%%%
\section{Jacobi Differential Operator with Varying Boundary Conditions}\label{s-examples}
%%%%%%%%%%%%%%%%%%%%%%%%%%%%%
%%%%%%%%%%%%%%%%%%%%%%%%%%%%% 

%%%%%%%%%%%%%%%%%%%%%%%%%%%%%
%%%%%%%%%%%%%%%%%%%%%%%%%%%%% 
\subsection{Boundary Triple and Boundary Pair Construction}
%%%%%%%%%%%%%%%%%%%%%%%%%%%%%
%%%%%%%%%%%%%%%%%%%%%%%%%%%%% 

Let $0<\al,\beta<1$, and consider the classical Jacobi differential expression given by
\begin{align}\label{e-jacobi}
\ell_{\al,\beta}[f](x)=-\dfrac{1}{(1-x)^{\al}(1+x)^{\beta}}[(1-x)^{\al+1}(1+x)^{\beta+1}f'(x)]'
\end{align}
on the maximal domain
\begin{align*}
\cD^{(\al,\beta)}\ti{max}=\{ f\in L^2_{\al,\beta}(-1,1) ~|~ f,f'\in AC\ti{loc};\ell_{\al,\beta}[f]\in L^2_{\al,\beta}(-1,1)\},
\end{align*}
where the Hilbert space 
\begin{align}\label{e-hilbertspace}
L^2_{\al,\beta}(-1,1):=L^2\left[(-1,1),(1-x)^{\al}(1+x)^{\beta}\right].
\end{align}
We denote the inner product on this space by $\langle \cdot , \cdot \rangle_{\alpha, \beta}$. This maximal domain defines the associated minimal domain given in Definition \ref{d-min} with defect indices $(2,2)$. The operator ${\bf L}\ti{min}$ will be used to denote the expression \eqref{e-jacobi} acting on the minimal domain. The specified values of $\al,\beta$ will ensure that the differential expression is in the limit-circle non-oscillating case at both endpoints, and so are assumed throughout. If either parameter is outside of the range, it changes the endpoint classification and results will need to be slightly modified accordingly. It is worth pointing out that the case $\al=\beta=0$ describes the Legendre differential equation.

The Jacobi polynomials $P_n^{(\al,\beta)}(x)$, $n\in\NN_0$, form a complete orthogonal set in $L^2_{\al,\beta}(-1,1)$ for which $P_n^{(\al,\beta)}(x)$ solves the eigenvalue equation of the symmetric expression given in equation \eqref{e-jacobi}, that is:
\begin{align}\label{e-jacobieigen}
\ell_{\al,\beta}[f](x)=n(n+\al+\beta+1)f(x)
\end{align}
for each $n\in \NN_0$. Both solutions to the equation $\ell_{\al,\beta}[f]=\lambda f$, $\la\in\CC$, are easily given in terms of hypergeometric functions, with the special case $\la=0$ given by $\widetilde{w}_1$ and $\widetilde{w}_2$ below.

The goal of the current subsection is to formulate all self-adjoint extensions of ${\bf L}\ti{min}$ via a rank-two perturbation, which is accomplished in Theorem \ref{t-jacobipert}. In order to make this characterization rigorous, some of the material in this subsection may feel heavy in explicit expressions. {\em Readers who are mainly interested in the spectral results may prefer to skip to Theorem \ref{t-jacobipert}, and return to the contents in this subsection later.}

Define the Gauss hypergeometric series (or function) as
\begin{align}\label{d-hyperseries}
F(a,b;c;z):=\,_2F_1(a,b;c;z)=F(b,a;c;z)=\sum_{k=0}^{\infty}\dfrac{(a)^{(k)}(b)^{(k)}}{(c)^{(k)}}\dfrac{z^k}{k!},\nonumber
\end{align}
where $(x)^k=x(x+1)(x+2)\cdots(x+k-1)$ denotes the Pochhammer function, or rising factorial, on the disk $|z|<1$. Notice that this notation is usually written as a subscript in the hypergeometric community. Define
\begin{equation*}
\begin{aligned}
\widetilde{w}_1(x):=&
\left.\begin{cases}
F(0,\al+\beta+1;\al+1;(1-x)/2) & \text{near }x=1 \\
-F(0,\al+\beta+1;\beta+1;(1+x)/2) & \text{near }x=-1
\end{cases}\right\}
=\left.\begin{cases}
1 & \text{near }x=1 \\
-1 & \text{near }x=-1
\end{cases}\right\}, \\
\widetilde{w}_2(x):=&\left.\begin{cases}
\dfrac{((1-x)/2)^{-\al}}{\al 2^{\al+\beta+1}}F(-\al,\beta+1;1-\al;(1-x)/2) & \text{near }x=1 \\
\dfrac{((1+x)/2)^{-\beta}}{\beta 2^{\al+\beta+1}}F(-\beta,\al+1;1-\beta;(1+x)/2) & \text{near }x=-1
\end{cases}\right\}.
\end{aligned}
\end{equation*}

The associated sesquilinear form is defined, for $f,g\in\cD\ti{max}^{(\al,\beta)}$, via equation \eqref{e-greens}. Integration by parts easily yields the explicit expression
\begin{align*}
[f,g](\pm 1):=\lim_{x\to\pm 1^{\mp}}(1-x)^{\al+1}(1+x)^{\beta+1}[g'(x)f(x)-f'(x)g(x)].
\end{align*}
Note that the dependence of the sesquilinear form on the parameters $\al$ and $\beta$ is suppressed in the definition for the sake of notation. Theorem \ref{t-limits} also says that the sesquilinear form is both well-defined and finite for all $f,g\in\cD\ti{max}^{(\al,\beta)}$. 

Two operations are then generated via these particular solutions:
\begin{align}\label{e-defineoperations}
     f^{[0]}(x):=[f,\widetilde{w}_1](x)  \text{ and } f^{[1]}(x):=[f,\widetilde{w}_2](x).\nonumber
\end{align}
Explicitly, we have
\begin{equation}\label{e-operations1}
\begin{aligned}
    f^{[0]}(-1)&=[f,\widetilde{w}_1](-1)=[f,-1](-1)=\lim_{x\to -1^+}-(1-x)^{\al+1}(1+x)^{\beta+1}f'(x), \\
    f^{[0]}(1)&=[f,\widetilde{w}_1](1)=[f,1](1)=\lim_{x\to 1^-}(1-x)^{\al+1}(1+x)^{\beta+1}f'(x), \\
    f^{[1]}(-1)&=[f,\widetilde{w}_2](-1)= \lim_{x\to -1^+}-f(x)-\frac{(1+x)f'(x)}{\beta}, \\
    f^{[1]}(1)&=[f,\widetilde{w}_2](1)= \lim_{x\to 1^-}f(x)-\frac{(1-x)f'(x)}{\al}.
\end{aligned}\nonumber
\end{equation}
If these operations are used to define two mappings

\begin{equation}\label{e-jacobibt}
\begin{aligned}
    \Gamma_0f:=\left(\begin{array}{c}
f^{[0]}(-1)  \\
f^{[0]}(1)
\end{array} \right), \hspace{.5cm}
    \Gamma_1f:=\left( \begin{array}{c}
f^{[1]}(-1)   \\
-f^{[1]}(1) 
\end{array} \right), \hspace{.5cm}
f\in\cD^{(\al,\beta)}\ti{max},
\end{aligned}
\end{equation}
then it can easily be shown that $\{\CC^2,\Gamma_0,\Gamma_1\}$ is a boundary triple for $\cD^{(\al,\beta)}\ti{max}$, as in \cite{F}. Additionally, the self-adjoint extension $\boldsymbol{A}_0$ will be chosen to correspond to the notation of Subsection \ref{ss-twolc} and have the domain 
\begin{align}\label{e-jacobi0}
    \dom \bA_0=\{f\in\cD^{(\al,\beta)}\ti{max} ~:~ f^{[1]}(-1)=f^{[1]}(1)=0 \}.
\end{align}

The slight abuse of notation here, with $\bA_0$ chosen to represent the boundary conditions given by $\ker(\Gamma_1)$, seems counterintuitive given the boundary triple but is a natural choice as $\bA_0$ will act as the unperturbed operator in our additive perturbation. Accordingly, another notation (albeit slightly abused, see Subsection \ref{ss-notation}) must be used to denote the self-adjoint operator with the boundary conditions given by $\ker(\Gamma_0)$, also inspired by perturbation theory:
\begin{align}\label{e-jacobiinfty}
    \dom \bA_{\infty}=\{f\in\cD^{(\al,\beta)}\ti{max} ~:~ f^{[0]}(-1)=f^{[0]}(1)=0 \}.\nonumber
\end{align}

The explicit matrix-valued Weyl $m$-functions for these two self-adjoint extensions can be written with the help of some special functions:
\begin{equation}\label{e-firstconstants}
\begin{aligned}
c_1(z)=&\frac{-\Gamma(\al+1)\Gamma(-\beta)}{\Gamma(-z+\al)\Gamma(z-\beta+1)}\,, \\
c_2(z)=&\frac{\beta 2^{\al+\beta+1}\Gamma(\al+1)\Gamma(\beta)}{\Gamma(z+1)\Gamma(-z+\al+\beta)}\,, \\
c_3(z)=&\frac{\Gamma(1-\al)\Gamma(-\beta)}{\al 2^{\al+\beta+1}\Gamma(-z)\Gamma(z-\al-\beta+1)}\,, \\
c_4(z)=&\frac{\beta\Gamma(1-\al)\Gamma(\beta)}{\al\Gamma(z-\al+1)\Gamma(-z+\beta)}\,,
\end{aligned}
\end{equation}
and our spectral parameter $\la$ is now given via the formula 
\begin{align}\label{e-Lambda}
    \la=
    (-z-1)(-z+\al+\beta)
    \quad \text{for }\lambda\in \mathbb{C}\setminus\mathbb{R},
\end{align}
due to the natural breakdown given in equation \eqref{e-jacobieigen}. Note this differs from \cite[Section 3]{F} only by substituting $\mu=-z-1$, which will be useful in the proof of Theorem \ref{t:mulambdan}. Otherwise, the functions $c_j$, for $j=1,\dots,4$, are the same as in \cite{F}. The Weyl $m$-function associated with the extension $\bA_{\infty}$ can be written as
\begin{align}
    M_{\infty}(\lambda)=\dfrac{1}{c_2(z)}
    \left( \begin{array}{cc}
-c_4(z) & 1 \\
1 & -c_1(z)
\end{array} \right)
\quad \text{for }\lambda\in \mathbb{C}\setminus\mathbb{R},\label{e-weylinfty}
\end{align}
and is naturally associated with the boundary triple $\{\CC^2,\Gamma_0,\Gamma_1\}$ for $\cD^{(\al,\beta)}\ti{max}$. Likewise, the Weyl $m$-function associated with the extension $\bA_0$ is
\begin{align}\label{e-weyl0}
    M_0(\lambda)=\dfrac{1}{c_3(z)}
    \left( \begin{array}{cc}
c_1(z) & 1 \\
1 & c_4(z)
\end{array} \right)
\quad \text{for }\lambda\in \mathbb{C}\setminus\mathbb{R}.
\end{align} 
\begin{Remark}\label{r-MGamma}
The function $M_0$ is determined by the choice of $\Gamma_0$. As we will see below Theorem \ref{t-GesTs}, the operations of $\Gamma_0$ fix the cyclic functions that define the spectral measure of our unperturbed operator $\bA_0$.$\hfill\bullet$
\end{Remark}
Interested readers should consult \cite[Section 3]{F} for additional details concerning the underlying solutions $\widetilde{w}_1$ and $\widetilde{w}_2$, the implemented boundary triples, and the associated Weyl $m$-functions. 

It is worth pointing out that these two $m$-functions can be obtained from one another by switching the operations $\Gamma_0$ and $\Gamma_1$. The effect of this switch is described in \cite[Corollary 2.5.4]{BdS}; explicitly we have $M_0=-(M_{\infty})^{-1}$ and $M_{\infty}=-(M_0)^{-1}$. In a general setup such as that of Subsection \ref{ss-twolc}, the Weyl $m$-functions are $2\times 2$ matrices so if one $m$-function is known it is usually not too difficult to compute the other. This will be of great help later.

Notice that the boundary triple $\{\CC^2,\Gamma_0,\Gamma_1\}$ for $\cD^{(\al,\beta)}\ti{max}$ given in equation \eqref{e-jacobibt} already matches that of Subsection \ref{ss-twolc}. Recall the standard conditions on principal and non-principal solutions $u_a$ and $v_a$ near a limit-circle endpoint $a$, respectively, from Subsection \ref{ss-twolc} and the normalizations already imposed on our functions. Given this, the statement in \cite[Lemma 6.10.10]{BdS} says that for $f\in\cD\ti{max}^{(\al,\beta)}$ we have
\begin{align*}
    \lim_{x\to a^+}\frac{f(x)}{v_a(x)}=-\lim_{x\to a^+}[f,u_a](x),
\end{align*}
and holds analogously at the other limit-circle endpoint $b$. This extra minus sign is already spoken for in the definition of $\widetilde{w}_1$ in the current example. 

Next, we set up a quadratic form $\ft$ (see equation \eqref{e-formdomainbase}) which will act on a domain $\fD$ (see equation \eqref{e-formdomain}). The function $\widetilde{w}_2$ is necessarily broken down into its actions near each endpoint as $\widetilde{w}_{2(-1)}$ and $\widetilde{w}_{2(1)}$. Fortunately, these functions are flexible in the sense that choosing values $c$ and $d$, in accordance with \eqref{e-intervals}, can be anything in our example. Additionally, $\widetilde{w}_2$ was chosen specifically so that $\la_0=0$. We arbitrarily let $c=-3/4$ and $d=3/4$ and notice that 
\begin{align*}
\sqrt{p(x)}=(1-x)^{\frac{\al+1}{2}}(1+x)^{\frac{\beta+1}{2}}\,,
\end{align*}
where $p(x)$ is from equation \eqref{e-jacobi} in accordance with Definition \ref{d-sturmdif}. Then, for all $f\in AC(-1,-3/4]$, we calculate that 
\begin{align*}
    N\ci{\widetilde{w}_{2(-1)}}f(x)=\sqrt{p(x)}f'(x)+\frac{2\beta(1-x)^{\frac{\al+1}{2}}(1+x)^{\frac{\beta+3}{2}}F\left(-\beta,\al+1;-\beta;\frac{1+x}{2}\right)}{F\left(-\beta,\al+1;1-\beta;\frac{1+x}{2}\right)}f(x),
\end{align*}
and for $f\in AC[3/4,1)$
\begin{align*}
    N\ci{\widetilde{w}_{2(1)}}f(x)=\sqrt{p(x)}f'(x)-\frac{2\al(1-x)^{\frac{\al+3}{2}}(1+x)^{\frac{\beta+1}{2}}F\left(-\al,\beta+1;-\al;\frac{1-x}{2}\right)}{F\left(-\al,\beta+1;1-\al;\frac{1-x}{2}\right)}f(x).
\end{align*}
A closed semi-bounded form is thus defined as 
\begin{align*}
    \ft[f,g]:=\int_{-1}^{-3/4} &\left(N\ci{\widetilde{w}_{2(-1)}}f(x)\right)\overline{\left(N\ci{\widetilde{w}_{2(-1)}}g(x)\right)}dx+\int_{3/4}^1\left(N\ci{\widetilde{w}_{2(1)}}f(x)\right)\overline{\left(N\ci{\widetilde{w}_{2(1)}}g(x)\right)}dx \\
    &+\int_{-3/4}^{3/4}\left((p(x)f'(x)\overline{g'(x)}+q(x)f(x)\overline{g(x)}\right)dx \\
    &+\dfrac{2\beta\left(\frac{7}{4}\right)^{\al+1}\left(\frac{1}{4}\right)^{\beta+2}F\left(-\beta,\al+1;-\beta;\frac{1}{8}\right)}{F\left(-\beta,\al+1;1-\beta;\frac{1}{8}\right)}f\left(-\frac{3}{4}\right)\overline{g\left(-\frac{3}{4}\right)} \\
    &-\dfrac{2\al\left(\frac{1}{4}\right)^{\al+2}\left(\frac{7}{4}\right)^{\beta+1}F\left(-\al,\beta+1;-\al;\frac{1}{8}\right)}{F\left(-\al,\beta+1;1-\al;\frac{1}{8}\right)}f\left(\frac{3}{4}\right)\overline{g\left(\frac{3}{4}\right)},
\end{align*}
for $f,g\in\fD$, where 
\begin{align*}
    \fD:=\Bigg\{f\in L^2_{\al,\beta}(-1,1)~:~f\in AC(-1,1)~, &\sqrt{p}f'\in L^2\left(-\frac{3}{4},\frac{3}{4}\right), \\
    &N\ci{\widetilde{w}_{2(-1)}}f\in L^2\left(-1,-\frac{3}{4}\right), N\ci{\widetilde{w}_{2(1)}}f\in L^2\left(\frac{3}{4},1\right)\Bigg\}.
\end{align*}
 
The mapping $\Gamma_0$ can now be extended to the form domain $\fD$
\begin{align*}
\Lambda f:=\left(\begin{array}{c}
f^{[0]}(-1) \\
f^{[0]}(1)
\end{array} \right)
=
\left(\begin{array}{c}
\lim_{x\to -1^+}-(1-x)^{\al+1}(1+x)^{\beta+1}f'(x) \\
\lim_{x\to 1^-}(1-x)^{\al+1}(1+x)^{\beta+1}f'(x)
\end{array} \right)
\quad \text{for } f\in\fD,
\end{align*}
so that $\{\CC^2,\Lambda\}$ is found to be a boundary pair for ${\bf L}\ti{min}$ compatible with the boundary triple $\{\CC^2,\Gamma_0,\Gamma_1\}$ given in equation \eqref{e-jacobibt}.

Finally, the explicit perturbation problem can be set up using this boundary pair structure. Define two weighted distributions that mimic the operation $\Lambda$ via
\begin{equation}\label{e-jacobidist}
\begin{aligned}
\langle f,\widetilde{\delta}_{-1}\rangle_{\alpha , \beta} &=\lim_{x\to -1^+}-(1-x)^{\al+1}(1+x)^{\beta+1}f'(x), \\
\langle f,\widetilde{\delta}_1\rangle_{\alpha, \beta} &=\lim_{x\to 1^-}(1-x)^{\al+1}(1+x)^{\beta+1}f'(x),
\end{aligned} \nonumber
\end{equation}
on the domain $\fD$. The inner product is always assumed to be that of the Hilbert space $L^2_{\al,\beta}(-1,1)$ (see equation \eqref{e-hilbertspace}). These distributions are immediately seen to be in $\cH_{-1}(\boldsymbol{A}_0)$ by \cite[Theorem 6.11.4]{BdS}. In accordance with Subsection \ref{ss-twolc}, a coordinate mapping ${\bf B}:\CC^2\to\Ran({\bf B})\subset\cH_{-1}\left({\bf A}_0\right)$ can be defined to act via multiplication by the row vector
\begin{align*}
    \begin{pmatrix}
    \widetilde{\delta}_{\alpha , \beta}, \widetilde{\delta}_{\alpha , \beta}
    \end{pmatrix}.
\end{align*}
Its adjoint ${\bf B}^*:\Ran({\bf B})\to\CC^2$ is given by
\begin{align*}
    {\bf B}^*f=
    \begin{pmatrix}
    \langle f,\widetilde{\delta}_{-1}\rangle_{1} \\
    \langle f,\widetilde{\delta}_1\rangle_{1}
    \end{pmatrix}.
\end{align*}
Note that although the inner product in these operations come from $L^2_{\al,\beta}(-1,1)$, the function $f$ is naturally restricted to the domain $\cH_1(\boldsymbol{A}_0)$, as was discussed in Remark \ref{r-pertforms}.

\begin{theo}\label{t-jacobipert}
Let $\Theta$ be a self-adjoint linear relation in $\CC^2$ and $\bA_0$ be the self-adjoint operator given in equation \eqref{e-jacobi0}. Define $\bA\ci\Theta$ as the singular rank-two perturbation:
\begin{align}\label{e-athetapert}
    \bA\ci\Theta:=\bA_0+{\bf B}\Theta{\bf B}^*.
\end{align}
Then every self-adjoint extension of the minimal operator ${\bf L}\ti{min}$ that acts via equation \eqref{e-jacobi} can be written as $\bA\ci\Theta$ for some $\Theta$. In particular, 
\begin{align}\label{e-domatheta}
    \dom\left(\bA\ci\Theta\right)=\left\{ f\in\cD\ti{max}^{(\al,\beta)}~:~\left\{\Gamma_0 f,\Gamma_1 f\right\}\in\Theta\right\}.
\end{align}
\end{theo}

\begin{proof}
The theorem follows as a direct application of Theorem \ref{t-twopert} to the Jacobi differential operator, with the setup given in the preceding discussion.
\end{proof}

%%%%%%%%%%%%%%%%%%%%%%%%%
%%%%%%%%%%%%%%%%%%%%%%%%%
\subsection{Spectral Theory of the Unperturbed Operator $\boldsymbol{A}_0$}\label{ss-unperturbspectral}
%%%%%%%%%%%%%%%%%%%%%%%%%
%%%%%%%%%%%%%%%%%%%%%%%%%

Realizing all self-adjoint extensions of the Jacobi differential operator through this perturbation setup allows for the tools of self-adjoint perturbation theory to be implemented. Next, we focus our attention on the matrix-valued spectral properties for the unperturbed operator $\bA_0$ so that this perturbation theory can be carried out.

The spectral properties of $\bA_0$ can be determined by analyzing $M_0$. Sturm--Liouville operators with two limit-circle endpoints have only discrete spectrum (see e.g.~\cite{BEZ, W, Z}), which are given by first order poles of the $m$-function \cite[Theorem 3.6.1(iv)]{BdS}. Explicitly, we will apply Theorem \ref{t-GesTs}.

The functions $c_j$, for $j=1,\dots,4$, given in \eqref{e-firstconstants}, can be analytically continued to wherever the Gamma functions do not have poles. By the restrictions of the parameters $\alpha, \beta\in (0,1)$, and the fact that the Gamma function does not have roots, we observe that all eigenvalues of $\bA_0$ are caused by the poles of the Gamma function $\Gamma(-z)$ in the denominator of $c_3$. Of course, it is well-known that the Gamma function has simple poles at the non-positive integers; when $z=n$, $n\in \NN_0$. By \eqref{e-Lambda}, the specific relation providing us with the location of the eigenvalues is:
\begin{align}\label{e-LambdaN}
    \la_n=
    (-n-1)(-n+\al+\beta)
    \quad \text{for }n\in\NN_0.
\end{align}

The functions $c_1$ and $c_4$ can be further simplified in this special case.

\begin{lem}\label{l:cs}
For $z=n \in \N_0$, the functions $c_1(n)$ and $c_4(n)$ reduce to
\begin{equation}\label{e-firstconstants2}
\begin{aligned}
c_1(z)&= \displaystyle (-1)^{n} \prod_{k=0}^{n}\frac{k-\al}{k- \beta}, \\
c_4(z)&= \displaystyle (-1)^{n} \prod_{k=0}^{n}\frac{k-\beta}{k- \al}.
\end{aligned}\nonumber
\end{equation}
\end{lem}

\begin{proof}
Using the recursive definition of the gamma function $\Gamma(z+1) = z\Gamma(z)$, we can simplify the factors of  
\[c_4(z)=\frac{\beta\Gamma(1-\al)\Gamma(\beta)}{\al\Gamma(z-\al+1)\Gamma(-z+\beta)} \] 
from \eqref{e-firstconstants} for $z=n \in \N_0$. First, we examine the factors containing $\al$,
\begin{align*}
    \frac{\Gamma(1-\al)}{\al\Gamma(n-\al+1)} &= -\frac{\Gamma(-\al)}{\Gamma(n-\al+1)}\\
    &= -\frac{\Gamma(-\al)}{(n-\al)\Gamma(n-\al)}\\
    &= -\frac{\Gamma(-\al)}{(n-\al) \dots (1-\al)(-\al) \Gamma(-\al)}\\
    &=-\displaystyle   \prod_{k=0}^{n} (k- \al)^{-1}.
\end{align*}
Then, those containing $\beta$,
\begin{align*}
    \frac{\beta \Gamma(\beta)}{\Gamma(-z+\beta)} &= \frac{ \beta (-1+\beta) \Gamma(-1+\beta)}{\Gamma(-n+\beta)} \\
    &= \frac{ \beta (-1+\beta) \dots (-(n+1)+ \beta ) (- n+ \beta ) \Gamma(- n+ \beta)}{\Gamma(-n+\beta)} \\
    &= \displaystyle \prod_{k=0}^{n} (-k+ \beta).
\end{align*}
Substituting these two into the expression for $c_4$ yields the desired result
\begin{align*}
    c_4(n) &= \displaystyle - \prod_{k=0}^{n}\frac{-k+ \beta}{k- \al}\\
    &= \displaystyle (-1)^{n} \prod_{k=0}^{n}\frac{k-\beta}{k- \al}.
\end{align*}
The computation for $c_1 (n)$ for $n \in \N_{0}$ follows from simply making identical substitutions with the roles of $\al$ and $\beta$ swapped, so we omit it.
\end{proof}

For a matrix-valued Borel measure $\mu: \mathcal{B}\to \C^{d\times d}_+$ (where $\mathcal{B}$ and $\C^{d\times d}_+$ denote the Borel sigma algebra of $\R$ and the set of positive definite $d\times d$ matrices, respectively) the space $L^2(\mu)$ arises from the inner product:
\begin{align}\label{d-matrixL^2}
\langle \f ,\psi \rangle\ci{L^2(\mu)}:=\int\ci\R \left\langle \dd\mu(x)\f(x),\psi(x) 
\right\rangle\ci{\C^d} 
\quad\text{for}\quad
\f(x)=\begin{pmatrix}\f_1(x)\\\vdots\\\f_d(x)\end{pmatrix}, \psi(x)=\begin{pmatrix}\psi_1(x)\\\vdots\\ \psi_d(x)\end{pmatrix}.
\end{align}

The following well-known result will allow us to extract spectral information from the Weyl $m$-function. 
\begin{theo}{\cite[Theorem 5.5(i)]{GT}}\label{t-GesTs}
For a matrix-valued Herglotz function with representation
\begin{align}\label{e-Mfunction}
    M(z) = C + Dz + \int_\mathbb{R} ((w-z)^{-1} - w(1+w^2)^{-1})d \mu(w),
\quad z\in \mathbb{C}_+,
\end{align}
we have
\begin{align*}
\mu\{\lambda\} = -i\lim_{\eps\searrow 0} \eps M(\lambda +i\eps).
\end{align*}
\end{theo}

Also, recall that a set $\{u_1,\hdots,u_k\}\subset\cH_{-1}(\boldsymbol{A}_0)$ is called cyclic for operator $\boldsymbol{A}_0$, if
\begin{align*}
    L^2_{\alpha,\beta}(-1,1) = \clos \spa \{(\boldsymbol{A}_0-\lambda I)^{-1}u_n:\lambda\in \CC\backslash\RR\text{ and } n\in\{1,\hdots,k\}\}.
\end{align*}
For now let $\mu$ be the matrix-valued spectral measure of $\boldsymbol{A}_0$ with respect to the cyclic set $\{\widetilde{\delta}_{1}, \widetilde{\delta}_{-1}\}$. Via Theorem \ref{t-GesTs} and Remark \ref{r-MGamma}, the measure $\mu$ (with this cyclic set) corresponds to the Weyl $m$-function stated in equation \eqref{e-weyl0}. The point spectrum of a self-adjoint operator, e.g.~ $\boldsymbol{T}$, will be denoted by $\sigma_p(\boldsymbol{T})$. While a detailed analysis of the perturbed operator $\bA_{\Theta}$ through its $m$-function $M_{\Theta}$ will start in Subsection \ref{ss-jacobipert}, it is necessary to introduce the following proposition which will be used to determine the multiplicity of eigenvalues even in this unperturbed setting.

\begin{prop}{\cite[Proposition 1]{DM2}}\label{p-dmmult}
Let $\{\CC^2,\Gamma_0,\Gamma_1\}$ be the boundary triple defined above, $\te$ be a Hermitian $2 \times 2$ matrix and  $z\in\rho(\bA_{\infty})$. Then
\begin{align*}
    z\in\sigma_p(\bA_{\Theta}) \iff 0\in\sigma_p(\Theta-M_{\infty}),
\end{align*}
and $\dim\Ker(\bA_{\Theta}-z)=\dim\Ker(\Theta-M_{\infty})$.
\end{prop}

The point masses of the matrix-valued weight of $\mu$ and multiplicity of eigenvalues for $\boldsymbol{A}_0$ can then be determined using the $m$-function from Theorem \ref{t-GesTs} and Proposition \ref{p-dmmult}. We will abuse notation slightly by not taking limits with respect to $\la=(-z-1)(-z+\al+\beta)$ but with respect to the corresponding $z$ within the calculation, as this is how  $M_0$ is defined.

\begin{theo}\label{t:mulambdan}
The matrix-valued weights of the point masses are
\begin{align*}
    \mu\{\lambda_n\} 
    &= 
    \frac{\al 2^{\al+\beta+1}(-1)^{n+1}\Gamma(n-\al-\beta+1)}{n!\Gamma(1-\al)\Gamma(-\beta)}\left( \begin{array}{cc}
c_1(n) & 1 \\
1 & c_4(n)
\end{array} \right)\,
.
\end{align*}
In particular, we verify the well-known fact that the multiplicity of each eigenvalue is one.
\end{theo}

\begin{proof}
By Theorem \ref{t-GesTs}, these matrix-valued weights are obtained by evaluating the limit of $\eps M(\lambda +i\eps)$ as $\eps\searrow 0$. The relations \eqref{e-Lambda}, \eqref{e-LambdaN} and \eqref{e-Mfunction}, together with the analyticity of the Gamma function away from its poles, yield
\begin{align*}
    \mu\{\lambda_n\} &= -i\lim_{\eps\searrow 0}\eps M_0(\lambda_n +i\eps)\\
    &=
    -i\lim_{\eps\searrow 0}\eps \dfrac{1}{c_3(n+i\eps)}
    \left( \begin{array}{cc}
c_1(n+i\eps) & 1 \\
1 & c_4(n+i\eps)
\end{array} \right)\\
    &= 
    \frac{-\al 2^{\al+\beta+1}\Gamma(n-\al-\beta+1)}{\Gamma(1-\al)\Gamma(-\beta)}\left( \begin{array}{cc}
c_1(n) & 1 \\
1 & c_4(n)
\end{array} \right)
\lim_{\eps\searrow 0}i\eps \Gamma(-n-i\eps)
.
\end{align*}

Using the well-known residues of the Gamma function, we obtain
\begin{align*}
    \lim_{\eps\searrow 0} i\eps \Gamma(-n +i\eps)
    =
    \text{Res}(\Gamma, -n) = \frac{(-1)^n}{n!}
    \quad\text{for all }n\in \mathbb{N}_0,
\end{align*}
proving the formula for the point masses.

It remains to show that the multiplicity of each eigenvalue equals one. Proposition \ref{p-dmmult} says that this is the case only if $\dim\Ker(M_{\infty}(\la_n))=1$ for each eigenvalue $\la_n$. It is enough to show that $M_{\infty}$ does not have rank zero or, equivalently, that it does not have nullity 2. Upon inspection, this occurs if and only if $1/c_2(n)=0$, which is impossible because the Gamma function has no zeros. We conclude that the multiplicity of each eigenvalue must be one. 
\end{proof}
% \begin{align}\label{e-conecfour}
% c_1(n)c_4(n)\neq 1.
% \end{align}
% For each $n\in \mathbb{N}_0$, we use the well-known identities $\Gamma(z)\Gamma(1-z) = \frac{\pi}{\sin(\pi z)}$ and $z\Gamma(z)=\Gamma(z+1)$, for $z\in \mathbb{C}\setminus \mathbb{Z}$, to find
% \begin{align*}
% c_1(n)c_4(n)
% &=
% \frac{-\beta\Gamma(\al+1)\Gamma(-\beta)\Gamma(1-\al)\Gamma(\beta)}{\al\Gamma(-z+\al)\Gamma(z-\beta+1)\Gamma(z-\al+1)\Gamma(-z+\beta)}\\
% &=
% \frac{-\sin((\al-n)\pi)\sin((\beta-n)\pi)}{\sin(\al\pi)\sin(-\beta\pi)}\,.
% \end{align*}
% Using the trigonometric identity $\sin(x+y)=\sin x \cos y + \cos x \sin y,$ as well as the facts that $\sin (-n\pi) = 0$ and $\cos (-n\pi) = (-1)^n$ for integers $n$, this simplifies to
% \begin{align*}
% c_1(n)c_4(n)
% &=\frac{-\sin(\al\pi)\cos(-n\pi)\sin(\beta\pi)\cos(-n\pi)}{\sin(\al\pi)\sin(-\beta\pi)}\\
% &=\frac{-\sin(\al\pi)\sin(\beta\pi)}{\sin(\al\pi)\sin(-\beta\pi)}\\
% &=\frac{-\sin(\beta\pi)}{\sin(-\beta\pi)}=1\,.
% \end{align*}

% Since $\alpha, \beta \in (0,1)$, equation \eqref{e-conecfour} is false for all $n\in \mathbb{N}_0$ and the multiplicity of each eigenvalue must be one.

%%%%%%%%%%%%%%%%%%%%%%%%%
%%%%%%%%%%%%%%%%%%%%%%%%%
\subsection{Perturbation Theory for the Jacobi Differential Operator}\label{ss-jacobipert}
%%%%%%%%%%%%%%%%%%%%%%%%%
%%%%%%%%%%%%%%%%%%%%%%%%%

The spectral measure of the unperturbed operator $\bA_0$ is now concrete thanks to Theorem \ref{t:mulambdan} and it is possible to explore spectral properties of other self-adjoint extensions. This will be done with a combination of techniques arising from the theory of boundary triples and finite rank self-adjoint perturbation theory. 

As we have seen in the previous subsection, the main tool for analyzing spectral properties through the theory of boundary triples is the Weyl $m$-function. The linear relation $\Theta$ from Theorem \ref{t-jacobipert}, which identifies a self-adjoint extension, can also be used to calculate the $m$-function for the extension. In particular, if $\Theta$ is a $2\times 2$ Hermitian matrix, \cite[Equation (3.8.7)]{BdS} says that for $\la\in\rho(\bA_{\infty})\cap\rho(\bA\ci\Theta)$ we have
\begin{align}\label{e-MTheta}
M\ci\Theta (\lambda)
=
(\Theta - M_{\infty}(\lambda))^{-1}.
\end{align}

Let $\Theta$ be a $2\times 2$ Hermitian matrix and write
\begin{equation*}
\Theta= 
\begin{pmatrix} 
\Theta_{11} & \Theta_{12}\\
\Theta_{21} & \Theta_{22}
\end{pmatrix}.
\end{equation*}
The use of equation \eqref{e-weylinfty} in \eqref{e-MTheta} yields
\begin{align}
\label{e-weyl5}
    M\ci\Theta (\lambda)&=(\Theta - M_{\infty}(\lambda))^{-1}
=\begin{pmatrix}
\Theta_{11}+\frac{c_4(z)}{c_2(z)} & \Theta_{12}-\frac{1}{c_2(z)} \\
\Theta_{21}-\frac{1}{c_2(z)} & \Theta_{22}+\frac{c_1(z)}{c_2(z)}
\end{pmatrix}^{-1}.
\end{align}

Many commonly studied special cases of boundary conditions, such as separated and periodic, fall into the category where $\Theta$ is a $2\times 2$ Hermitian matrix, see e.g.~\cite[Example 6.3.6]{BdS} and \cite[Section 5.1]{F} for some examples.

\begin{Remark}\label{r-0vsinfty}
It is worth briefly describing how  $\Theta$ here corresponds to the perturbation in Theorem \ref{t-jacobipert} and clarifying some points of possible confusion. The theory of boundary triples has a natural self-adjoint extension whose domain is given by functions in the intersection of the maximal domain and the $\ker(\Gamma_0)$. This extension is $\bA_{\infty}$ in our example. Hence, $M_{\infty}$ is the starting point of many formulas that come from boundary triples. In contrast, the unperturbed operator in our construction is $\bA_0$, necessary by Remark \ref{r-formordering}, so we would expect to start with $M_0$ to determine spectral properties. In this sense, perturbation theory and boundary triples are starting with ``opposite'' extensions, so to speak. 

Fortunately, this discrepancy is not a problem. A quick sanity check uncovers that by plugging $\Theta=\left\{\CC^2,0\right\}$, the $0$ matrix, into equation \eqref{e-MTheta} the $m$-function should be $-M_{\infty}^{-1}=M_0$. Plugging this $\Theta$ into equation \eqref{e-athetapert} for our perturbation also yields $\bA_0$. \hfill $\bullet$
\end{Remark}

For linear relations that are not $2\times 2$ matrices, there is a similar relationship that can be exploited. If $\Theta$ is a self-adjoint relation in $\CC^2$ then there exist $2\times 2$ matrices $\cA$ and $\cB$ such that 
\begin{align}\label{e-relationdecomp}
    \cA^*\cB=\cB^*\cA, \hspace{4mm} \cA\cB^*=\cB\cA^*, \hspace{4mm} \cA^*\cA+\cB^*\cB=\cI=\cA\cA^*+\cB\cB^*
\end{align}
and
 \begin{align*}
     \Theta=\left\{\left\{\cA\f,\cB\f\right\}~:~\f\in\CC^2\right\},
 \end{align*}
 by \cite[Corollary. 1.10.9]{BdS}. The domain of the self-adjoint extension $\bA\ci\Theta$ given in equation \eqref{e-domatheta} then has the special form
 \begin{align*}
     \dom\left(\bA\ci\Theta\right)=\left\{ f\in\cD\ti{max}^{(\al,\beta)}~:~\cA^*\Gamma_1 f=\cB^*\Gamma_0 f\right\}.
 \end{align*}
 Finally, a general analog of equation \eqref{e-MTheta} can be written for linear relations (see \cite[Equation 3.8.4]{BdS}:
 \begin{align*}%\label{e-Mlinearrelation}
 M\ci\Theta(\la)=\left(\cA^*+\cB^*M_{\infty}(\la)\right)\left(\cB^*-\cA^*M_{\infty}(\la)\right)^{-1}.
 \end{align*}
 
The case where $\G$ is a linear relation is unfortunately not covered in the following Theorems due to the calculational complexity introduced by the above formulas. Indeed, the case where $\G$ is a Hermitian $2 \times 2$ matrix is much more applicable. If there arises a situation where the spectrum of an extension is needed for a particular linear relation then the following results can still serve as a road map for extracting such information.

An eigenvalue $\la_n^\Theta\in\rho(\bA_{\infty})$ of $\boldsymbol{A}\ci\Theta$ is said to be degenerate if for the corresponding $z_n\in\RR$ such that $\la_n^\Theta=(-z_n-1)(-z_n+\al+\beta)$ we have 
\begin{align}\label{e-degenerate}
    \Theta_{11}=-\dfrac{c_4(z_n)}{c_2(z_n)}, \quad \Theta_{21}=\dfrac{1}{c_2(z_n)} \quad\text{ and }\quad \Theta_{22}=-\frac{c_1(z_n)}{c_2(z_n)}.
\end{align}
Note that for $\la_n^\Theta\in\rho(\bA_{\infty})$, these values are all well-defined thanks to $1/c_2(z_n)$ not having a pole and that each of the two possible choices of $z_n$ necessarily yield the same $\la_n$. Degenerate eigenvalues thus correspond to the case where $\Theta-M_{\infty}$ results in the zero matrix. An operator $\bA\ci\Theta$ is called non-degenerate if all of its eigenvalues are non-degenerate.

\begin{cor}\label{c-degen}
If $\Theta$ is a $2\times 2$ Hermitian matrix and $\la_n^\Theta$ is a degenerate eigenvalue of $\boldsymbol{A}\ci\Theta$, then it has multiplicity two.
\end{cor}

\begin{proof}
The result immediately follows from Proposition \ref{p-dmmult} and the above observation that degenerate eigenvalues correspond to the case where $\Theta-M_{\infty}$ is the zero matrix. 
\end{proof}

When eigenvalues are non-degenerate, we now show that their multiplicity must be one, sharpening results mentioned in \cite[Chapter 10.11]{Z}.

\begin{theo}\label{t:simple}
If $\Theta$ is a $2\times 2$ Hermitian matrix or $\Theta=\{0\}\times\CC^2$, then non-degenerate eigenvalues of the self-adjoint extension $\boldsymbol{A}\ci\Theta$ defined by equation \eqref{e-athetapert} have multiplicity one.
\end{theo}

\begin{proof}
Any self-adjoint extension of ${\bf L}\ti{min}$ is given by a self-adjoint relation $\Theta$ in $\CC^2$ via Theorem \ref{t-twopert}. We cover the two distinct cases in the hypotheses of the Theorem separately:
\begin{itemize}
    \item[(i)] $\Theta$ is a $2\times 2$ Hermitian matrix or
    \item[(ii)] $\Theta=\{0\}\times\CC^2$.
\end{itemize}
Case $(i)$: 
The standard formula for the inverse of a $2\times 2$ matrix, yields the following simplification of equation \eqref{e-weyl5}
\begin{align*}
    M\ci\Theta (\lambda)&=(\Theta - M_{\infty}(\lambda))^{-1}
=\begin{pmatrix}
\Theta_{11}+\frac{c_4(z)}{c_2(z)} & \Theta_{12}-\frac{1}{c_2(z)} \\
\Theta_{21}-\frac{1}{c_2(z)} & \Theta_{22}+\frac{c_1(z)}{c_2(z)}
\end{pmatrix}^{-1}\nonumber
=K(z)L(z)
\quad\text{for }\lambda\in \mathbb{C}\setminus\mathbb{R},\\
\intertext{where $\la=
    (-z-1)(-z+\al+\beta)$
     by equation \eqref{e-Lambda},}
K(z) &=\dfrac{c_2(z)}{\left(\Theta_{11}c_2(z)+c_4(z)\right)\left(\Theta_{22}c_2(z)+c_1(z)\right)-\left(\Theta_{12}c_2(z)-1\right)\left(\Theta_{21}c_2(z)-1\right)}\\
\intertext{and}
L(z) &=
\begin{pmatrix}
\Theta_{22}c_2(z)+c_1(z) & 1-\Theta_{12}c_2(z) \\
1-\Theta_{21}c_2(z) & \Theta_{11}c_2(z)+c_4(z)
\end{pmatrix}.
\end{align*}
By Theorem \ref{t-GesTs}, the operator $\boldsymbol{A}\ci{\Theta}$ has eigenvalues at those $\la$ that are poles of $M\ci\Theta(\la)$. 

Now recall the definitions of $c_1(z)$, $c_2(z)$ and $c_4(z)$ and that the Gamma function does not have roots on the real line. Further observe that at the poles of the Gamma functions that arise in the denominators of $c_1(z)$, $c_2(z)$ and $c_4(z)$, the expression of $M\ci\Theta(\la)$ remains bounded. 

Therefore, the poles of $M\ci\Theta(\la)$ can only occur at those points where the denominator of $K(z)$ vanishes:
\begin{align}\label{e-thetaevas}
    \left(\Theta_{11}c_2(z)+c_4(z)\right)\left(\Theta_{22}c_2(z)+c_1(z)\right)-\left(\Theta_{12}c_2(z)-1\right)\left(\Theta_{21}c_2(z)-1\right)=0.
\end{align}
This denominator is just $\det L(z)$ so we conclude that the rank of $\Theta-M_{\infty}(\lambda)$ is one for a $\la$ whose corresponding value of $z$ satisfies equation \eqref{e-thetaevas}. The eigenvalue $\la$ was also assumed to be non-degenerate so the nullity of $\Theta-M_{\infty}(\lambda)$ is thus one as well. Proposition \ref{p-dmmult} then says that the multiplicity of $\la$ is one. All eigenvalues are discrete because the operator has two limit-circle endpoints \cite{BEZ,Z}.

Case $(ii)$: This occurs when $\boldsymbol{A}\ci\Theta$ is the Friedrichs extension of ${\bf L}\ti{min}$, as the form domain of this self-adjoint extension (see Remark \ref{r-pertforms}) is given by 
\begin{align*}
    \dom(\ft\ci\Theta)=\left\{f\in\fD~:~\Lambda f=0\right\},
\end{align*}
The linear relation $\Theta$ thus chooses the transversal extension to the starting operator $\boldsymbol{A}_0$ due to the way our boundary pair was defined.
\end{proof}

While the previous Theorem limits the multiplicity of non-degenerate eigenvalues, it does not say anything about the frequency with which degenerate eigenvalues occur. The following Corollary sheds some light on this question.

\begin{cor}
Let $\la\in(-\al-\beta,\infty)\backslash[\sigma(\bA_0)\cup\sigma(\bA_{\infty})]$. Then $\la$ is a degenerate eigenvalue for some $\bA\ci\Theta$ where $\Theta$ is a Hermitian $2\times 2$ matrix. 
\end{cor}

\begin{proof}
The lowest eigenvalue of $\bA_0$ is $-\al-\beta$ and it is easy to see that any $\la>-\al-\beta$ corresponds to two values of $z$ via the relationship $\la=(-z-1)(-z+\al+\beta)$. For $\la\in\rho(\bA_0)\cup\rho(\bA_{\infty})$ it suffices to take one such value of $z$ and generate a Hermitian matrix $\Theta$ via the equations in \eqref{e-degenerate}. Choosing such a $\Theta$ will yield an $\bA\ci\Theta$ with $\la$ as a degenerate eigenvalue by definition.
\end{proof}

Of course, not every choice of $\Theta$ yields a degenerate eigenvalue $\la$; there are many ways to satisfy equation \eqref{e-thetaevas}.

%%%%%%%%%%%%%%%%%%%%%%%%%%%%%
%%%%%%%%%%%%%%%%%%%%%%%%%%%%% 
\subsection{Eigenfunctions and the Spaces $L^2(\mu^\Theta)$}\label{ss-Eigen}
%%%%%%%%%%%%%%%%%%%%%%%%%%%%%
%%%%%%%%%%%%%%%%%%%%%%%%%%%%% 

Let us first consider the unperturbed operator $\boldsymbol{A}_0$ and its spectral representation given by the operator $\boldsymbol{M}$ on $L^2(\mu)$ (the operator that multiplies by the independent variable $\boldsymbol{M}:L^2(\mu)\to L^2(\mu)$ with $\boldsymbol{M}f(x) = xf(x)$). The magnitude of the eigenvector does not matter, so we simply dropped the multiplicative constant from Theorem \ref{t:mulambdan} and normalized the first component to equal one.  
\begin{cor}\label{c:l2mu}
The eigenvector of $\boldsymbol{A}_0$ corresponding to eigenvalue $\la_n$ takes the following form in the spectral representation
\begin{align*}
f_n(x)=\chi\ci{\{\la_n\}}(x)\begin{pmatrix} 1\\c_4(n)\end{pmatrix}\in L^2(\mu),
\end{align*}
where $\chi\ci{\{\lambda_n\}}$ denotes the characteristic function at $x=\la_n$. We obtain an explicit formulation of the space $L^2(\mu)=\clos \spa \left\{f_n:n\in \N_0\right\}.$
\end{cor}

Amusingly, although we only know the location of the eigenvalues implicitly, the directions of the eigenfunctions do not depend on $\lambda_n$ but merely on $n$; so these directions are stated explicitly.

\begin{proof}
The direction of $f_n$ can be obtained by simply reading off the direction of Ran$\,\mu\{\lambda_n\}$ from the representation in Theorem \ref{t:mulambdan}. Since the operator $\boldsymbol{A}_0$ has only eigenvalues $\lambda_n$, the eigenfunction corresponding to $\lambda_n$ takes on this direction at the locations of the eigenvalues and is zero away from $\lambda_n$. Of course, the eigenfunction is unique only up to multiplication by a non-zero scalar.

Operator $\boldsymbol{A}_0$ is densely defined on the Hilbert space $L^2_{\al,\beta}(-1,1)$. Therefore, the span of the eigenfunctions in the spectral representation is complete in the space $L^2(\mu)$.
\end{proof}

Next, we apply these ideas to the perturbed operators $\boldsymbol{A}\ci\Theta$. Recall that by the spectral theorem the operator $\boldsymbol{A}\ci{\Theta}$ is unitarily equivalent to the multiplication by the independent variable on a vector-valued space $L^2(\mu^\Theta)$ with inner product defined analogously to equation \eqref{d-matrixL^2}. Here $\mu^\Theta$ is the matrix-valued spectral measure of $\boldsymbol{A}\ci{\Theta}$ with respect to the cyclic set $\{\widetilde{\delta}_{1}, \widetilde{\delta}_{-1}\}$. Note that $\mu^\Theta$ is defined by fixing this cyclic set; this is analogous to the definition of $\mu$ that is given after Theorem \ref{t-GesTs}. 
The definition of $M\ci\Theta$ through equation \eqref{e-MTheta} and the definition of $\bA\ci\Theta$ in equation \eqref{e-athetapert} are also fixed through this same choice of the cyclic vectors. Consequently, $\mu^\Theta$ and $M\ci\Theta$ indeed align (i.e.~they are related via Theorem \ref{t-GesTs}).

The space $L^2(\mu^\Theta)$ is vector-valued and Theorem \ref{t:simple} informs us that the eigenspace is one-dimensional in the non-degenerate case (i.e.~when \eqref{e-degenerate} does not hold). Noting that simple minor adjustments suffice for degenerate eigenvalues, we focus on the non-degenerate case in the next corollary and discuss the degenerate case in Remark \ref{r-degenerate1} below. So for now, assume that all eigenvalues of $\boldsymbol{A}\ci\Theta$ are simple. In the language of the spectral representation, this means that $\mu^\Theta\{x\}$ has rank one whenever $x=\lambda_n^\Theta$ (for all $n\in\N_0$) and rank zero for all other $x$. Here we express the eigenvectors in the spectral representation in terms of the eigenvalues $\lambda^\Theta_n$, which are implicitly given by equation \eqref{e-thetaevas}. It is worth noting that the main expression of \eqref{e-thetaevas} is thought to be transcendental, and therefore exact locations of eigenvalues may not be described explicitly.

Carrying out the analogous proof to that of Corollary \ref{c:l2mu}, while replacing the use of Theorem \ref{t:mulambdan} by Theorem \ref{t:simple}, yields the following result:

\begin{cor}\label{c:EVE2}
Let $\Theta$ be a Hermitian $2\times 2$ matrix corresponding to non-degenerate operator $\boldsymbol{A}\ci\Theta$. The eigenvector $f^\Theta_n\in L^2(\mu^\Theta)$ of $\boldsymbol{A}\ci{\Theta}$ corresponding to eigenvalue $\la^{\Theta}_n$ takes the following form in the spectral representation $$f^{\Theta}_n(x) =\chi\ci{\{\lambda^{\Theta}_n\}}(x)\begin{pmatrix} 1-\Theta_{12}c_2(z_n)\\\Theta_{11}c_2(z_n)+c_4(z_n)\end{pmatrix}.$$
The space $L^2(\mu^\Theta)=\clos \spa \left\{f_n^\Theta:n\in \N_0\right\}.$
\end{cor}

\begin{Remark}
At first sight, it seems surprising that only two of the entries of $\Theta$ occur in the eigenvector. Upon reflection of how we arrived at the corollary, it becomes clear that this is because the eigenvalue is located exactly where the other column of the matrix is co-linear by Theorem \ref{t:simple}. $\hfill\bullet$
\end{Remark}

\begin{Remark}
\label{r-degenerate1}
If $\lambda_n^\Theta$ is a degenerate eigenvalue, then the corresponding eigenspace in the spectral representation is spanned by $\chi\ci{\{\lambda^{\Theta}_n\}}(x)\begin{pmatrix} 1\\0\end{pmatrix}\in L^2(\mu^\Theta)$ and $ \chi\ci{\{\lambda^{\Theta}_n\}}(x)\begin{pmatrix} 0\\1\end{pmatrix}\in L^2(\mu^\Theta).$ $\hfill\bullet$
\end{Remark}

From \cite[Lemma 2.5]{LT_JST} we know that $\Ran(\bB)$ is cyclic for $\boldsymbol{A}\ci{\Theta}$. Some additional information about the location of the eigenvalues can be obtained by applying the following theorem to perturbations of $\bA_0$.

\begin{theo}[{\cite[Theorem 6.3]{LT_JST}}]\label{t:countable}
Let $\G_t=\G_0 + t\G$, where $\G$, $\G_0$ are Hermitian $2\times 2$ matrices and $\G>0$. Let $\tr\mu^{\G_0+t\G}$ be the scalar spectral measures of $\boldsymbol{A}\ci{\Theta}$. For an arbitrary singular Radon measure $\nu$ on $\R$, 
\[
\nu \perp \tr\mu^{\G_0+t\G}\ti s
\] 
for all except maybe countably many $t\in\R$. 
\end{theo}

Setting $\nu = \tr\mu^{\widetilde\Theta}$ for any Hermitian $2\times 2$ matrix $\widetilde\Theta$, allows us to restrict the relative location of eigenvalues of the operators ${\bf A}\ci\Theta$ along certain lines in the space of perturbation parameters:

\begin{prop}\label{p:mutsing}
Let $\boldsymbol{A}\ci{\Theta}$, $\G_0$, $\G$ and $\G_t$  be defined as in Theorem \ref{t:countable}. For any Hermitian $2\times 2$ matrix $\widetilde\Theta$, the spectra $\sigma\left(\boldsymbol{A}\ci{\widetilde\Theta}\right)\cap\sigma\left(\boldsymbol{A}\ci{\Theta_t}\right)=\varnothing$ for all but possibly countably many $t\in\R$.
\end{prop}

\begin{proof}
In Theorem \ref{t:countable} we choose $\nu = \tr\mu^{\widetilde\Theta}$. Since the $\boldsymbol{A}\ci{\Theta}$ has only eigenvalues, $ \nu$ is a singular Radon measure.  By the conclusion of Theorem \ref{t:countable}, we obtain that $\tr\mu^{\widetilde\Theta}\perp\tr\mu^{\G_t}\ti{s}$ for all but possibly countably many $t$. Since the eigenvalues of $\boldsymbol{A}\ci{\widetilde\Theta}$ and $\boldsymbol{A}\ci{\Theta_t}$ are discrete (see e.g.~\cite{BEZ, AZBook, Z}), the set of eigenvalues equals these operators' spectra, respectively.
\end{proof}

%%%%%%%%%%%%%%%%%%%%%%%%%
%%%%%%%%%%%%%%%%%%%%%%%%%
\subsection{Image of $f_0$ under $V\ciG$}\label{ss-V}
%%%%%%%%%%%%%%%%%%%%%%%%%
%%%%%%%%%%%%%%%%%%%%%%%%%

We include a simple yet instructive application of the Representation Theorem \ref{t:repr-01} onto the first eigenfunction $f_0$ from Corollary \ref{c:EVE2}. In principle, we can replace $f_0$ by any given function $f\in L^2(\mu).$ We chose $f_0$ just to keep expressions simple.

We recall {\cite[Theorem 5.1]{LT_JST}}, which provides us with a formula for the spectral representation of $\boldsymbol{A}\ci{\Theta}$, i.e.~for the unitary operator $V\ciG: L^2(\mu)\to L^2(\mu^\G)$ intertwining $\boldsymbol{A}\ci{\Theta}$ and $\boldsymbol{M}$ for $\G$ a Hermitian $2\times 2$ matrix
\[
V\ciG \boldsymbol{A}\ci{\Theta} = \boldsymbol{M} V\ciG,
\]
which maps each column of $\boldsymbol{B}$ to itself. 

\begin{theo}
\label{t:repr-01}
The spectral representation $V\ciG$ takes the form

\begin{align}
\label{Repr-01}
\left(V\ci\ciG h
\be\right)(t) 
=
h(t) \be
-
\G\int\ci\R \frac{h(x)-h(t)}{x-t} [d\mu(x)] \be \nonumber
\end{align}

for $\be\in \C^d$ and compactly supported $h\in C^1(\R)$.
\end{theo}

Notice that for $h\in C^1(\R),$ the integral does not have a singularity by the Mean Value Theorem. 

\begin{cor}\label{t-eigentrace}
When $\Theta$ is such that $\boldsymbol{A}\ci\Theta$ is non-degenerate and $\sigma\left(\boldsymbol{A}_0\right)\cap\sigma\left(\boldsymbol{A}\ci{\Theta}\right)=\varnothing$, we have
\begin{align*}
    \left(V\ci\ciG f_0\right)(t)&= \frac{ -2^{\al+\beta+1}(\al^2+\beta^2)\Gamma(1-\al-\beta)}{(\alpha+\beta+t)\Gamma(1-\al)\Gamma(-\beta)}
\begin{pmatrix}
\Theta_{12}/\alpha + \Theta_{11}/\beta\\
\Theta_{22}/\alpha + \Theta_{21}/\beta
\end{pmatrix}.
\end{align*}
In particular, the image under $V\ciG$ of the first eigenvector of $\boldsymbol{A}$ is being smeared out over all the eigenvectors of $\boldsymbol{A}\ciG$.
\end{cor}

\begin{Remark}
For the interpretation of this function in $L^2(\mu^\Theta)$ only the values of $(V\ci\ciG  f_0)(t)$ at $t=\lambda^\Theta_n$ matter. Further, a scalar factor will occur from taking the orthogonal projection of the vector in $(V\ci\ciG  f_0)(\lambda_n^\Theta)$ onto the direction of the vector in $f_n^\Theta$ from Corollary \ref{c:EVE2}.$\hfill\bullet$
\end{Remark}

\begin{proof}
The location of the zeroth ($n=0$) eigenfunction $f_0$ from Corollary \ref{c:l2mu} is obtained from \eqref{e-Lambda}:
$
\lambda_0 = -\alpha-\beta.
$
Moreover, Lemma \ref{l:cs} in conjunction with Corollary \ref{c:l2mu} informs us that
\[
\alpha f_0(x) = \widetilde f(x):= \chi\ci{\{-\alpha-\beta\}}(x)
\begin{pmatrix}
\alpha\\\beta
\end{pmatrix}.
\]
In order to apply Theorem \ref{t:repr-01}, we need to use a compactly, once differentiable function. To that end, let
$\be = \begin{pmatrix}
\alpha\\\beta
\end{pmatrix}$ 
and where is some $h(x)\in C^1$ which equals $1$ at $s_0 = -\alpha-\beta$ and drops to zero appropriately to the left and right of $s_0$. This necessitates two assumptions about the function $h(x)$: Without loss of generality (because the spectrum of $A$ is discrete), we assume that
\begin{align}
    \label{e:condition2}
    h(\lambda_n) &= 0
    \qquad\text{for all }n\in \N.
\intertext{and}
\label{e:condition1}
    h(\lambda_n^\Theta) &= 0
    \qquad\text{for all }n\in \N_0.
\end{align}

In $L^2(\mu)$ we have $\widetilde f_0(x) = h(x) \be$. So, we obtain
\begin{align*}
    \left(V\ci\ciG \widetilde f\right)(t) =\left(V\ci\ciG h
\be\right)(t)
&=
h(t) \be
-
\G\int\ci\R \frac{h(x)-h(t)}{x-t} [d\mu(x)] \be
=
-
\G\int\ci\R \frac{h(x)}{x-t} [d\mu(x)] \be,
\end{align*}
upon realizing \eqref{e:condition1} and using Corollary \ref{c:EVE2}. 

Recall that $\mu$ is the matrix-valued atomic measure given in Theorem \ref{t:mulambdan}. Now, by Condition \eqref{e:condition2}, function $h(x)$ is only non-zero at the atom of the zeroth eigenvalue, when $x = x_0 =-\alpha-\beta$. Therefore, we obtain
\begin{align*}
    \left(V\ci\ciG \widetilde f\right)(t)&=-\G\int\ci\R \frac{h(x)}{x-t} [d\mu(x)] \be
= \frac{1}{\alpha+\beta+t}\G \mu\{\lambda_0\} \be
\\
& = \frac{-\al 2^{\al+\beta+1}\Gamma(1-\al-\beta)}{(\alpha+\beta+t)\Gamma(1-\al)\Gamma(-\beta)}
\begin{pmatrix}
\Theta_{11} & \Theta_{12}\\
\Theta_{21} & \Theta_{22}
\end{pmatrix}
\left( \begin{array}{cc}
\alpha/\beta & 1 \\
1 & \beta/\alpha
\end{array} \right)
\begin{pmatrix}
\alpha\\\beta
\end{pmatrix}
\\
& = \frac{-\al (\alpha^2+\beta^2) 2^{\al+\beta+1}\Gamma(1-\al-\beta)}{(\alpha+\beta+t)\Gamma(1-\al)\Gamma(-\beta)}
\begin{pmatrix}
\Theta_{12}/\alpha + \Theta_{11}/\beta\\
\Theta_{22}/\alpha + \Theta_{21}/\beta
\end{pmatrix}.
\end{align*}
The lemma follows from $ f_0(x) = \widetilde f(x)/\alpha$ and the linearity of $V\ciG$.
\end{proof}

\begin{Remark}
 The condition $\sigma\left(\boldsymbol{A}_0\right)\cap\sigma\left(\boldsymbol{A}\ci{\Theta}\right)=\varnothing$ is a bit stronger than necessary. In the proof we only use \eqref{e:condition2} and \eqref{e:condition1}. Proposition \ref{p:mutsing} provides us with sufficient (but by no means necessary) conditions for many $\Theta$.$\hfill\bullet$
\end{Remark}

\begin{Remark}
We note that it is also possible to start with any function $h{\bf e}\in L^2(\mu^{\Theta_1})$, with $h\in C^1(\R)$ compactly supported, by using $V\ci{\Theta_2-\Theta_1}:L^2(\mu^{\Theta_1})\to L^2(\mu^{\Theta_2})$ in Theorem \ref{t:repr-01}.$\hfill\bullet$
\end{Remark}

%%%%%%%%%%%%%%%%%%%%%%%%%%%%%
%%%%%%%%%%%%%%%%%%%%%%%%%%%%%
{\bf Acknowledgements.}
%%%%%%%%%%%%%%%%%%%%%%%%%%%%%
%%%%%%%%%%%%%%%%%%%%%%%%%%%%%

The authors would like to thank Pavel Kurasov and Annemarie Luger for many insightful conversations regarding the construction of the perturbation and the use of boundary triples and pairs.

\begin{Backmatter}

%%%%%%%%%%%%%%%%%%%%%%%%%%%%%
%%%%%%%%%%%%%%%%%%%%%%%%%%%%%

\printaddress

\end{Backmatter}

\end{document}